\documentclass[a4paper]{amsart}

\usepackage[mathscr]{eucal}
\usepackage{amsmath,amssymb,amsfonts}
\usepackage{mathrsfs,latexsym,amsthm,enumerate}
\newtheorem{theorem}{Theorem}[section]
\newtheorem{lemma}[theorem]{Lemma}
\newtheorem{corollary}[theorem]{Corollary}
\newtheorem{example}[theorem]{Example}

\newtheorem{proposition}[theorem]{Proposition}
\newtheorem{remark}[theorem]{Remark}

\title{AF inverse monoids and the structure of countable MV-algebras}

\author{Mark V. Lawson}
\address{Mark V.Lawson, Department of Mathematics
and the
Maxwell Institute for Mathematical Sciences, 
Heriot-Watt University,
Riccarton,
Edinburgh EH14 4AS, 
UNITED KINGDOM}
\email{m.v.lawson@hw.ac.uk}

\author{Philip Scott}
\address{Philip Scott,
Department of Mathematics and Statistics, 
University of Ottawa,
585 King Edward,
Ottawa, Ontario K1N 6N5, 
CANADA}
\email{phil@site.uottawa.ca}

\thanks{The first author was partially supported by EPSRC grant EP/I033203/1, 
and also by his stay at the University of Ottawa as a Distinguished Research Visitor in April and May of 2012.
He would like to thank his colleagues at the Mathematics Department in Ottawa for their great hospitality. 
The second author was partially supported by an NSERC Discovery Grant. The work was primarily
done while on sabbatical in 2014 at the School of Informatics,
University of Edinburgh. He would like to thank Gordon Plotkin and Alex Simpson for their hospitality
and support. He would also like to thank Bart Jacobs (Nijmegen) for many stimulating discussions.
Both authors would like to thank Daniele Mundici for his comments.}

\begin{document} 

%%%%%%%%%%%%%%%%%%%%%%%%%%%%%%%%%%%%%%%%%%%%%%%%%%%%%%%%%%%%%%%%%%%%%%%%%%%%%%%%%%%%%%%%%%%%%%
\begin{abstract} We define a class of inverse monoids having the property that their lattices of principal ideals naturally form an MV-algebra.
We say that an arbitrary MV-algebra can be co-ordinatized if it is isomorphic to one constructed in this way from such a monoid.  
We prove that every countable MV-algebra can be so co-ordinatized.
The particular inverse monoids needed to establish this result are examples of what we term AF inverse monoids.
These are constructed from Bratteli diagrams and arise naturally as direct limits of finite direct products of finite symmetric inverse monoids.
\end{abstract}
\maketitle

%%%%%%%%%%%%%%%%%%%%%%%%%%%%%%%%%%%%%%%%%%%%%%%%%%%%%%%%%%%%%%%%%%%%%%%%%%%%%%%%%%%%%%%%%%%%%%%%%
\section{Introduction}

MV-algebras were introduced by C. C . Chang in 1958 \cite{Chang}.
In Chang's  original axiomatization, it is plain that such algebras are generalizations of Boolean algebras.
In general, the elements of an MV-algebra are not idempotent, but those that are form a Boolean algebra.
A good introduction to their theory may be found in Mundici's tutorial notes \cite{M3}.
The standard reference is \cite{CDM}.
The starting point for our paper is Mundici's own work that connects countable MV-algebras to a class of AF $C^{\ast}$-algebras \cite{M1,MP}.
He sets up a correspondence between AF $C^{\ast}$-algebras whose Murray-von Neumann order is a lattice and countable MV-algebras.
In \cite{M2}, he argues that AF algebras `should be regarded as sort of noncommutative Boolean algebras'.
This is persuasive because the commutative AF $C^{\ast}$-algebras are function algebras over separable Boolean spaces.
But the qualification `sort of' is important.
The result would be more convincing if commutative meant, precisely, countable Boolean algebra.
In this paper, we shall introduce a class of countable structures whose commutative members are precisely this.

Approximately finite (AF) $C^{\ast}$-algebras,  that is those $C^{\ast}$-algebras which are direct limits of finite dimensional $C^{\ast}$-algebras,  were introduced by Bratteli in 1972 \cite{Bratteli},
and form one of the most important classes of $C^{\ast}$-algebras.
Reading Bratteli's paper, it quickly becomes apparent that his calculations rest significantly on the properties of matrix units.
The reader will recall that these are square matrices all of whose entries are zero except for one place where the entry is one.
Our key observation is that matrix units of a given size $n$ form a groupoid, 
and this groupoid determines the structure of a finite symmetric inverse monoid on $n$ letters.
The connection is via what are termed rook matrices \cite{Solomon}.
Symmetric inverse monoids are simply the monoids of all partial bijections of a given set.
Here the set can be taken to be $\{1, \ldots, n\}$.
These monoids have a strong Boolean character.
For example, their semilattices of idempotents form a finite Boolean algebra.
They are however non-commutative.
This leads us to define a general class of Boolean inverse monoids, called AF inverse monoids, constructed from Bratteli diagrams.
We argue that this class of monoids is the most direct non-commutative generalization of Boolean algebras.
For example, they figure in the developing theory of non-commutative Stone dualities \cite{Law3,Law4,Law5,LL1,LL2} where they are associated with a class of \'etale topological groupoids.
Significantly, commutative AF inverse monoids are countable Boolean algebras.
It is worth noting that the groups of units of such inverse monoids have already been studied \cite{DM,KS,LN} though sans the inverse monoids.

We prove that the poset of principal ideals of an AF inverse monoid naturally forms an MV-algebra when that poset is a lattice.
Accordingly, we say that an MV-algebra that is isomorphic to an MV-algebra constructed in this way
may be co-ordinatized by an inverse monoid.
The main theorem we prove in this paper is that {\em every} countable MV-algebra may be co-ordinatized in this way.
As a concrete example, we provide an explicit description of the AF inverse monoid that co-ordinatizes the MV-algebra of dyadic rationals in the unit interval.
It turns out to be a discrete version of the CAR algebra.
Finally, our results can be viewed as contributing to the study of the poset of $\mathscr{J}$-classes of an inverse semigroup.
For results in this area and further references, see \cite{Meakin}.
There are also {\em thematic} links between our work and that to be found in \cite{Ara,GW,Wehrung}.
This has influenced our choice of terminology when referring to partial refinement monoids.

%%%%%%%%%%%%%%%%%%%%%%%%%%%%%%%%%%%%%%%%%%%%%%%%%%%%%%%%%%%%%%%%%%%%%%%%%%%%%%%%%%%%%%%%%%%%%
\section{Basic definitions}

We shall work with two classes of structures in this paper: inverse monoids and partial commutative monoids.
In addition  to defining the structures we shall be working with, we shall also define precisely what we mean by co-ordinatizing an MV-algebra by means of an inverse monoid.

%%%%%%%%%%%%%%%%%%%%%%%%%%%%%%%%%%%%%%%%%%%%%%%%%%%%%%%%%%%%%%
\subsection{Boolean inverse monoids}

For inverse semigroup theory, we refer the reader to \cite{Law1}.
However, we need little theory {\em per se}, rather a number of definitions and some very basic examples.
Recall that an {\em inverse semigroup} is a semigroup $S$ in which for each element $s$ there is a unique element $s^{-1}$
satisfying $s = ss^{-1}s$ and $s^{-1} = s^{-1}ss^{-1}$.
Inverse semigroups are well-equipped with idempotents since both $s^{-1}s$ and $ss^{-1}$ are idempotents.
The set of idempotents of $S$ is denoted by $E(S)$ and is always a commutative idempotent subsemigroup.
For this reason, it is usually referred to as  the {\em semilattice of idempotents} of $S$.
Our inverse semigroups will always have a zero and ultimately an identity and so will be monoids.
If $a$ is an element of an inverse semigroup such that $e = a^{-1}a$ and $f = aa^{-1}$,
then we shall often write $e \stackrel{a}{\longrightarrow} f$ or $e\, \mathscr{D}\, f$
and say that $e$ is the {\em domain} of $a$ and $f$ is the {\em range} of $a$.
Accordingly, we define $\mathbf{d}(a) = a^{-1}a$ and $\mathbf{r}(a) = aa^{-1}$.
We define $s \, \mathscr{D} \, t$ if  $\mathbf{d}(s) \, \mathscr{D} \, \mathbf{d}(t)$, and $s \, \mathscr{J}\, t$ if $SsS = StS$. 
Observe that $\mathscr{D} \subseteq \mathscr{J}$.
These are the only two of Green's relations needed in this paper.

Three relations definable on any inverse semigroup will play significant r\^oles.
The {\em natural partial order} $\leq$ is defined by $a \leq b$ if $a = be$ for some idempotent $e$.
Despite appearances it is ambidextrous,
compatible with the multiplication, 
and $a \leq b$ implies that $a^{-1} \leq b^{-1}$.
Observe that if $a,b \leq c$ then both $a^{-1}b$ and $ab^{-1}$ are idempotents.
This leads to the definition of our second relation.
The {\em compatibility relation} $\sim$ is defined as follows:
$a \sim b$ if $a^{-1}b$ and $ab^{-1}$ are both idempotents.
Thus $a \sim b$ is a necessary condition for $a$ and $b$ to have a join.
A subset is said to be compatible if any two elements in the subset are compatible.
Finally, there is a refinement of the compatibility relation.
Elements of an inverse semigroup $a$ and $b$ are said to be {\em orthogonal}, denoted by $a \perp b$, if $a^{-1}b = 0 = ab^{-1}$.

The particular classes of inverse monoids we shall use in this paper may now be defined.
An inverse monoid with zero is said to be {\em distributive} if its semilattice of idempotents is a distributive lattice,
all compatible binary joins exist, and multiplication distributes over binary joins.
A distributive inverse monoid is said to be {\em Boolean} if its lattice of idempotents is a Boolean algebra.
Morphisms of distributive inverse monoids are monoid homomorphisms that map zero to zero and which preserve binary compatible joins.
More about Boolean and distributive inverse monoids can be found in  \cite{Law5,LL1,LL2}.
An inverse semigroup in which all binary meets exist is called a {\em $\wedge$-monoid.}

The key examples of inverse monoids needed to understand this paper are the following.
The {\em symmetric inverse monoids} $I(X)$ are the monoids of all partial bijections of the set $X$.
When $X$ has $n$ elements, we denote the corresponding symmetric inverse monoid by $I_{n}$.
We call the elements of $X$ {\em letters}.
Symmetric inverse monoids are Boolean inverse $\wedge$-monoids.
We define an inverse monoid to be {\em semisimple} if it is isomorphic to a finite direct product of finite symmetric inverse monoids.
They will play an important r\^ole in this paper.
We also mention here two properties that arise naturally in our work.
An inverse semigroup is said to be {\em fundamental} if the only elements that commute with every idempotent are themselves idempotents.
An inverse monoid is said to be {\em factorizable} if every element is beneath an element in the group of units.
Symmetric inverse monoids are fundamental and the finite ones are also factorizable.
Thus semisimple inverse monoids are factorizable and fundamental.
Our use of the word `semisimple' was motivated by the theory of $C^{\ast}$-algebras and the following theorem.
See \cite{Law5} for a proof.

\begin{theorem}\label{them:nice} 
The finite fundamental Boolean inverse $\wedge$-monoids are precisely the semisimple inverse monoids.
\end{theorem}

{\em This paper is based on an exact analogy between semisimple inverse monoids and finite dimensional $C^{\ast}$-algebras.}\\

We conclude this section with some useful properties of meets and joins in inverse semigroups.
The proofs of (1) and (2) may be found in \cite{Law1}, whereas (3) is folklore.

\begin{lemma}\label{lem:meets-joins} \mbox{}
\begin{enumerate}

\item  We have that $s \sim t$ if and only if $s \wedge t$ exists and $\mathbf{d}(s \wedge t) = \mathbf{d}(s) \wedge \mathbf{d}(t)$ and  $\mathbf{r}(s \wedge t) = \mathbf{r}(s) \wedge \mathbf{r}(t)$

\item In a distributive inverse monoid, if $a \vee b$ exists we have that
$$\mathbf{d}(a \vee b) = \mathbf{d}(a) \vee \mathbf{d}(b)
\mbox{ and }
\mathbf{r}(a \vee b) = \mathbf{r}(a) \vee \mathbf{r}(b).$$

\item  In a distributive inverse monoid,
if $a \vee b$ and $c \wedge (a \vee b)$ both exist
then $c \wedge a$ and $c \wedge b$ both exist, 
the join $(c \wedge a) \vee (c \wedge b)$ exists and
$$c \wedge (a \vee b)
=
(c \wedge a) \vee (c \wedge b).$$

\end{enumerate}
\end{lemma}

%%%%%%%%%%%%%%%%%%%%%%%%%%%%%%%%%%%%%%%%%%%%%%%%%%%%%%%%%%%%%%%%%%%%%%%%%%%%%%%%%%%%%%%%%%
\subsection{Partial refinement monoids}

Terminology in the area of partial algebras is not as well established as that of classical algebra.
Moreover, the two areas of dimension theory and effect (and MV) algebras have often developed
their own terminology for similar structures.
We have opted to use mainly the terminology of dimension theory \cite{GW,Wehrung} 
augmented by terminology from the theories of effect and MV-algebras to be found in \cite{BF, DP, J, JP, M1, M2, M3, MP}.  
See also \cite{Jacobs} for a modern categorical treatment of effect algebras.

Let $E$ be a set equipped with a partially defined operation denoted $\oplus$ together with a constant $0$.
If $a \oplus b$ is defined we write $\exists a \oplus b$.
We define a {\em partial commutative monoid} to be such a set satisfying the following three axioms.
\begin{itemize}
\item[{\rm (E1)}] $a \oplus b$ is defined if, and only if, $b \oplus a$ is defined and then they are equal.

\item[{\rm (E2)}]  $(a \oplus b) \oplus c$ is defined if, and only if,  $a \oplus (b \oplus c)$ is defined and then they are equal.

\item[{\rm (E3)}] For all $a \in E$, $\exists a \oplus 0$ and $a \oplus 0 = a$. 
\end{itemize}

A partial commutative monoid that satisfies, in addition, the following axiom
\begin{itemize} 
\item[{\rm (E4)}] The {\em refinement property\footnote{Also called the Riesz Decomposition Property (RDP). See \cite{Pul} and \cite{DP}.}}: \\ if $a_{1} \oplus a_{2} = b_{1} \oplus b_{2}$ then
there exist elements $c_{11},c_{12},c_{21},c_{22}$ such that
$a_{1} = c_{11} \oplus c_{12}$ and $a_{2} = c_{21} \oplus c_{22}$,
and
$b_{1} = c_{11} \oplus c_{21}$ and $b_{2} = c_{12} \oplus c_{22}$.
\end{itemize}
is called a {\em partial refinement monoid}.
A partial commutative monoid that satisfies, in addition, the following axiom
\begin{itemize}
\item[{\rm (E5)}] If $a \oplus b = 0$ then $a = 0$ and $b = 0$.
\end{itemize}
is said to be {\em conical} (or {\em positive}).
A partial commutative monoid that satisfies, in addition, the following axiom
\begin{itemize}
\item[{\rm (E6)}] If $a \oplus b = a \oplus c$ then $b = c$. 
\end{itemize}
is said to be {\em cancellative}.
To state the following two axioms, we need an additional constant denoted by 1.
\begin{itemize}
\item[{\rm (E7)}] $a \oplus 1$ is defined if, and only if, $a = 0$.

\item[{\rm (E8)}] For each $a \in E$, there exists a unique $a' \in E$ such that $a \oplus a'$ exists and equals $1$.
\end{itemize}

An {\em effect algebra} $(E, \oplus, 0, 1)$ is a structure satisfying (E1), (E2), (E7) and (E8).
For the following see, for example, \cite{FB1994}.

\begin{lemma} 
In an effect algebra, the axioms (E3),  (E5),  and (E6) hold automatically.
\end{lemma}

Define $a \leq b$ if, and only if, $b = a \oplus c$ for some $c$.
In an effect algebra $(E, \oplus, 0, 1)$, this relation is a partial order.   
A lattice-ordered effect algebra that also satisfies axiom (E4), the refinement property, is called an {\em MV-algebra} \cite{FB1994,Foulis2000}.
In an MV-algebra, there is an everywhere defined binary operation
$$a \boxplus b = a \oplus (a' \wedge b).$$
It is possible to axiomatize MV-algebras in terms of this operation \cite{Foulis2000}. 
MV algebras arose in the algebraic foundations of many-valued logics (\cite{Chang,M3}).

%%%%%%%%%%%%%%%%%%%%%%%%%%%%%%%%%%%%%%%%%%%%%%%%%%%%%%%%%%%%%%%%%%%%%%%%%%%%%%%%%%%%%%%%%%%%%%%%%%%%%%%%%%%%%%%%%%%%%%%%%%
\subsection{Co-ordinatization}

In this section, we shall define precisely what we mean by co-ordinatizing an MV-algebra by an inverse monoid.
The idea behind our construction can be found sketched on page 131 of \cite{Ren}. 
It is also related to the notion of coordinatizing a continuous geometry in the sense of von Neumann \cite{vN1,vN2}.

An inverse monoid is said to be {\em completely semisimple}\footnote{There is no connection with the term `semisimple' we introduced earlier.} 
if $e \, \mathscr{D} \, f \leq e$ implies $e = f$ for any idempotents $e$ and $f$. 
In completely semisimple inverse monoids, we have that $\mathscr{D} = \mathscr{J}$.
If $e \in E(S)$, a Boolean algebra, we denote by $\bar{e}$ the complement of $e$.
We say that $\mathscr{D}$ {\em preserves complementation} if $e \, \mathscr{D} \, f$ implies that  $\overline{e} \, \mathscr{D} \, \overline{f}$.

\begin{lemma}\label{lem:properties} A Boolean inverse monoid in which $\mathscr{D}$ preserves complementation is factorizable.
\end{lemma}
\begin{proof} Let $a \in S$.
Put $e = a^{-1}a$ and $f = aa^{-1}$.
Then $e \, \mathscr{D} \, f$.
By assumption, $\overline{e} \, \mathscr{D} \, \overline{f}$.
Thus there is an element $b$ such that $\overline{e} \stackrel{b}{\longrightarrow} \overline{f}$.
The elements $a$ and $b$ are orthogonal and so have a join $g = a \vee b$.
But $g^{-1}g = 1 = gg^{-1}$ and so $g$ is an invertible element and, by construction, $a \leq g$.
Thus $S$ is factorizable.
\end{proof}

\begin{lemma}\label{lem:clocks} A Boolean inverse monoid in which $\mathscr{D}$ preserves complementation
and in which $e \, \mathscr{D}\,  1$ implies that $e = 1$ is completely semisimple.
\end{lemma}
\begin{proof}
Suppose that $e \, \mathscr{D} \, f \leq e$.
Let $e \stackrel{a}{\longrightarrow} f$.
Observe that $afa$ has domain $e$.
Let $\overline{e} \stackrel{b}{\longrightarrow} \overline{f}$.
The elements $afa$ and $b$ are orthogonal.
We may therefore form the orthogonal join $b \vee afa$.
Its domain is $e \vee \overline{e} = 1$.
By assumption, its range must also be the identity.
Therefore $\overline{f} \vee afa^{-1} = 1$.
This is an orthogonal join in a Boolean algebra and so $f = afa^{-1}$.
That is, $(af)(af)^{-1} = f$.
But $af \leq a$ and $a$ has range $f$.
We deduce that $a = af$.
Thus $a^{-1}a = a^{-1}af$ and so $e = ef$.
But $f = ef$ and so $e = f$, as required. 
\end{proof}

Let $S$ be an arbitrary Boolean inverse monoid.
Put $\mathsf{E}(S) = E(S)/\mathscr{D}$.
We denote the $\mathscr{D}$-class containing the idempotent $e$ by $[e]$.
Define 
$[e] \oplus [f]$ as follows.
Suppose that we can find idempotents $e' \in [e]$ and $f' \in [f]$ such that $e'$ and $f'$ are orthogonal.
Then define $[e] \oplus [f] = [e' \vee f']$.
Otherwise, the operation $\oplus$ is undefined.

\begin{proposition}\label{prop:E} Let $S$ be a Boolean inverse monoid.
\begin{enumerate}

\item The operation $\oplus$ defined on $\mathsf{E}(S)$ is well-defined.

\item  $(\mathsf{E}(S),\oplus,[0],[1])$ is a conical partial refinement monoid satisfying (E7).

\item $[e] \leq [f]$ if, and only if, $e \, \mathscr{D}\, i \leq f$ for some idempotent $i$.

\item If the partial algebra satisfies (E6), cancellativity, then the inverse monoid is completely semisimple.

\item If $\mathscr{D}$ preserves complementation and $e \, \mathscr{D} \,  1$ implies that $e = 1$, then the partial algebra is an effect algebra.

\item The construction $S \mapsto \mathsf{E}(S)$ is functorial.

\end{enumerate}
\end{proposition}
\begin{proof} (1) Let 
$e' \, \mathscr{D} \, e''$
and 
$f' \, \mathscr{D} \, f''$
where $e'$ is orthogonal to $f'$, and $e''$ is orthogonal to $f''$.
We prove that $e' \vee f' \, \mathscr{D} \, e'' \vee f''$.
By assumption, there are elements $e' \stackrel{a}{\longrightarrow} e''$ and $f' \stackrel{b}{\longrightarrow} f''$.
The elements $a$ and $b$ are orthogonal and so $a \vee b$ exists.
But $e' \vee f' \stackrel{a \vee b}{\longrightarrow} e'' \vee f''$. 

(2) It is immediate that (E1) holds.

To prove axiom (E2), ironically, takes a bit of work.
Suppose that 
$\exists ([e] \oplus [f]) \oplus [g].$ 
Then 
$\exists [e] \oplus [f]$
and so we may find
$e \stackrel{a}{\longrightarrow} e'$ and $f \stackrel{b}{\longrightarrow} f'$
such that $e'$ and $f'$ are orthogonal.
By definition, $[e] \oplus [f] = [e' \vee f']$.
Since $\exists [e' \vee f'] \oplus [g]$,
we may find $e' \vee f' \stackrel{c}{\longrightarrow} i$
and
$g \stackrel{d}{\longrightarrow} g'$
such that $i$ and $g'$ are orthogonal.
It follows that 
$$([e] \oplus [f]) \oplus [g] = [i \vee g'].$$

Define $x = ce'$ and $y = cf'$.
Then
$$e' \stackrel{x}{\longrightarrow} \mathbf{r}(x)
\mbox{ and }
f' \stackrel{y}{\longrightarrow} \mathbf{r}(y).$$
Since $i$ is orthogonal to $g'$ and $\mathbf{r}(y) \leq i$,
we have that $\mathbf{r}(y)$ and $g'$ are orthogonal.
In addition, $yb$ has domain $f$ and range $\mathbf{r}(y)$.
It follows that $\exists [f] \oplus [g]$ and it is equal to $[\mathbf{r}(y) \vee g']$.
Observe next that $\mathbf{r}(x)$ is orthogonal to $\mathbf{r}(y)$ and, since $\mathbf{r}(x) \leq i$ it is also orthogonal to $g'$.
It follows that $\mathbf{r}(x)$ is orthogonal to $\mathbf{r}(y) \vee g'$.
But $xa$ has domain $e$ and range $\mathbf{r}(x)$.
It follows that
$\exists [e] \oplus [\mathbf{r}(y) \vee g']$ is defined
and equals 
$[\mathbf{r}(x) \vee \mathbf{r}(y) \vee g']$.
But $\mathbf{r}(x) \vee \mathbf{r}(y) = i$.
It follows that we have shown
$$\exists [e] \oplus ([f] \oplus [g])$$
and that it equals
$([e] \oplus [f]) \oplus [g]$.

The reverse implication follows by symmetry.

It is immediate that (E3) holds.

We prove that (E4) holds.
Let $[e_{1}] \oplus [e_{2}] = [f_{1}] \oplus [f_{2}]$ where we assume, without loss of generality,
that $e_{1}$ is orthogonal to $e_{2}$, and $f_{1}$ is orthogonal to $f_{2}$.
Let $e_{1} \vee e_{2} \stackrel{x}{\longrightarrow} f_{1} \vee f_{2}$.
Clearly
$$x = (f_{1} \vee f_{2})x(e_{1} \vee e_{2}).$$
Put
$$x_{1} = f_{1}xe_{1}, \quad 
x_{2} = f_{1}xe_{2}, \quad
x_{3} = f_{2}xe_{1}, \quad
x_{4} = f_{2}xe_{2}.$$
Then
$$x = x_{1} \vee x_{2} \vee x_{3} \vee x_{4},$$
an orthogonal join.
Define also
$$a_{11} = [\mathbf{d}(x_{1})], \quad
a_{12} = [\mathbf{d}(x_{2})], \quad
a_{21} = [\mathbf{d}(x_{3})], \quad
a_{22}= [\mathbf{d}(x_{4})].$$
Observe that $\mathbf{d}(x_{1}), \mathbf{d}(x_{3}) \leq e_{1}$.
Thus 
$\mathbf{d}(x_{1}) \vee \mathbf{d}(x_{3}) = e_{1}$
and
$\mathbf{d}(x_{2}) \vee \mathbf{d}(x_{4}) = e_{2}$.
Thus $a_{11} \oplus a_{21} = [e_{1}]$ and $a_{12} \oplus a_{22} = [e_{2}]$.
Similarly,
$f_{1} = \mathbf{r}(x_{1}) \vee \mathbf{r}(x_{2})$
and
$f_{2} = \mathbf{r}(x_{3}) \vee \mathbf{r}(x_{4})$.
Thus
$[f_{1}] = a_{11} \oplus a_{12}$ 
and
$[f_{2}] = a_{21} \oplus a_{22}$.

(E5) holds because if the join of two idempotents is 0 then both idempotents must be 0,
and the only idempotent $\mathscr{D}$-related to 0 is 0 itself.

(E7) holds because the only idempotent orthogonal to the identity is 0,
and the only idempotent $\mathscr{D}$-related to 0 is 0 itself.

(3) Suppose that 
$e \stackrel{x}{\longrightarrow} i \leq f$.
We may find an idempotent $j$ such that $f = i \vee j$ and $i \wedge j = 0$.
Then $[e] \oplus [j] = [f]$ and so $[e] \leq [f]$.
Conversely, 
suppose that $[e] \leq [f]$ where $e$ and $f$ are idempotents.
Then there exists an idempotent $g$ such that $[e] \oplus [g] = [f]$.
By definition, there are elements $e \stackrel{a}{\longrightarrow} e'$ and $g \stackrel{b}{\longrightarrow} g'$
such that $e' \vee f' \, \mathscr{D} \, f$.
But then $e \, \mathscr{D} \, e' \leq f$, as required.

(4) Suppose that $e \, \mathscr{D} \, f \leq e$.
Then $e = f \vee e \setminus f$.
Put $g = e \setminus f$.
Thus $[e] = [e] \oplus [g]$.
But then $[e] \oplus [0] = [e] \oplus [g]$.
By cancellation, we get that $[0] = [g]$ and so $g = 0$ from which we deduce that $e = f$, as required.

(5) Define $[e]' = [\overline{e}]$.
This is well-defined since $\mathscr{D}$ preserves complementation.
Clearly, $[e] \oplus [\overline{e}] = [1]$.
Suppose that $[e] \oplus [f] = [1]$.
By definition, we have idempotents $i$ and $j$ such that
$i \, \mathscr{D} \, e$ and $j \, \mathscr{D} \, f$ and $i \vee j \, \mathscr{D} \, 1$.
By assumption, $i \vee j = 1$, an orthogonal join.
It follows that $j = \overline{i}$.
But $i \, \mathscr{D} \, e$ implies that $\overline{i}\,  \mathscr{D} \, \overline{e}$.
Thus $[f] = [e]'$, as required.

(6) Let $\theta \colon S \rightarrow T$ be a morphism of Boolean inverse monoids.
Any morphism preserves the $\mathscr{D}$-relation and so we may define
$\theta^{\ast} \colon \mathsf{E}(S) \rightarrow \mathscr{E}(T)$ by $\theta^{\ast}([e]) = [\theta (e)]$.
Suppose that $[e] \oplus [f]$ is defined.
Then there exist idempotents $e'$ and $f'$ such that
$e \, \mathscr{D} \, e'$ and $f \, \mathscr{D} \, f'$ and where $e'$ and $f'$ are orthogonal.
Thus $[e] \oplus [f] = [e' \vee f']$.
But orthogonality is preserved by morphisms and $\theta (e' \vee f') = \theta (e') \vee \theta (f')$.
It follows that $\theta^{\ast}([e] \oplus [f]) = \theta^{\ast}([e]) \oplus \theta^{\ast}([f])$.
Morphisms are also morphisms of Boolean algebras and so $\theta^{\ast}([e]') = \theta^{\ast}([e])'$.
It is now straightforward to check that we have actually defined a functor from Boolean inverse monoids to partial algebras.
\end{proof}

We shall call $\mathsf{E}(S)$ equipped with the partially defined binary operation $\oplus$ the {\em partial algebra} associated with the inverse monoid $S$.
By Lemma~\ref{lem:clocks} and Proposition~\ref{prop:E},  we immediately deduce the following.

\begin{corollary}\label{cor:lurgi} Let $S$ be a Boolean inverse monoid that is completely semisimple
and in which  $\mathscr{D}$ preserves complementation.
Then the partial algebra $(\mathsf{E}(S),\oplus,[0],[1])$ is an effect algebra satisfying the refinement property.
\end{corollary}

In the light of the above result, it is convenient to define a {\em Foulis monoid}
to be a completely semisimple Boolean inverse monoid in which  $\mathscr{D}$ preserves complementation.
The construction $S \mapsto \mathsf{E}(S)$ is in fact a functor from the category of Foulis monoids to the category of effect algebras with the refinement property.
Since a Foulis monoid is completely semisimple, $\mathscr{D} = \mathscr{J}$.
It follows that we may identity $\mathsf{E}(S)$ with $S/\mathscr{J}$, the poset of principal ideals of $S$.
We say that an effect algebra $E$ can be {\em co-ordinatized} if there is a Foulis monoid $S$ such that $E$ is isomorphic to $S/\mathscr{J}$ as an effect algebra.
An inverse monoid with zero $S$ is said to satisfy the {\em lattice condition} if $S/\mathscr{J}$ is a lattice.
If $S$ is a Foulis monoid satisfying the lattice condition then $S/\mathscr{J}$ is in fact an MV-algebra.
Using the definitions we have made, the goal of this paper can now be precisely stated:\\

{\em For each countable MV-algebra $E$, show that there is a Foulis monoid $S$ satisfying the lattice condition such that $E$ is isomorphic to $S/\mathscr{J}$ as an MV-algebra.}\\

%%%%%%%%%%%%%%%%%%%%%%%%%%%%%%%%%%%%%%%%%%%%%%%%%%%%%%%%%%%%%%%%%%%%%%%%%%%%%%%%%%%%%%%%%%%%%%%%%%%%%%%%
\subsection{Co-ordinatizing finite MV-algebras}

In this section, we shall prove that all finite MV-algebras are co-ordinatizable,
a result motivating everything we do subsequently.
But we also take the opportunity to introduce some ideas that will also be useful to us later.

Let $S$ be a Boolean inverse monoid.
The following definition was suggested by \cite{CSGH}.
An {\em invariant mean} for $S$ is a function $\mu \colon E(S) \rightarrow [0,1]$ such that
the following axioms hold:
\begin{description}

\item[{\rm (IM1)}] $\mu (1) = 1$.

\item[{\rm (IM2)}] For any $s \in S$, we have that $\mu (s^{-1}s) = \mu (ss^{-1})$.

\item[{\rm (IM3)}] If $e$ and $f$ are orthogonal idempotents we have that $\mu (e \vee f) = \mu (e) + \mu (f)$.

\end{description}
Observe that since $0$ is orthogonal to itself $\mu (0) = 0$.

We shall say that an invariant mean is {\em good} if for all $e,f \in E(S)$ if $\mu (e) \leq \mu (f)$
then there exists $e'$ such that $\mu (e) = \mu (e')$ and $e' \leq f$.
This definition is adapted from \cite{ADMY}.
Finally, we say that an invariant mean {\em reflects the $\mathscr{D}$-relation} if $\mu (e) = \mu (f)$ implies that $e\, \mathscr{D} \, f$.

\begin{lemma}\label{lem:5April} Let $S$ be a Boolean inverse monoid equipped with a good invariant mean $\mu$  that reflects the $\mathscr{D}$-relation.
Then $S$ is completely semisimple, $\mathscr{D}$ preserves complementation, and $S/\mathscr{J}$ is linearly ordered.
\end{lemma}
\begin{proof} Observe first that if $e$ and $f$ are any idempotents such that $e \leq f$ then $\mu (e) \leq \mu (f)$.
If $e \neq f$, we may suppose that $e < f$.
Then $f = e \vee (f \overline{e})$, an orthogonal join.
Thus $\mu (f) = \mu (e) + \mu (f \overline{e})$.
It follows that $\mu (f) \geq \mu (e)$.
Next observe that $\mu (e) = 0$ implies that $e = 0$.
We have that $\mu (e) = \mu (0)$ and so, by assumption, $e \, \mathscr{D} \, 0$.
It is now immediate that $e = 0$.

We show that $S$ is completely semisimple.
Suppose that $e \, \mathscr{D} \, f \leq e$.
Then $\mu (e) = \mu (f)$.
Since $f \leq e$, we may write $e = f \vee (e \overline{f})$, an orthogonal join.
Thus $\mu (e \overline{f}) = 0$ and so, by the above, $e \overline{f} = 0$.
It follows that $e = f$, as required.

We show that $\mathscr{D}$ preserves complementation.
Suppose that $e \, \mathscr{D} \, f$.
Then $\mu (e) = \mu (f)$.
We have that 
$\mu (\bar{e}) = 1 - \mu (e)$
and
$\mu (\bar {f}) = 1 - \mu (f)$.
Thus $\mu (\bar{e}) = \mu (\bar{f})$.
Hence, by assumption, $\bar{e} \, \mathscr{D} \, \bar{f}$, as required.

Finally, we show that $S/\mathscr{J}$ is linearly ordered thus, in particular, $S$ satisfies the lattice condition.
Let $e$ and $f$ be arbitrary idempotents.
Without loss of generality, we may assume that $\mu (e) \leq \mu(f)$.
Since the invariant mean $\mu$ is good,
there is an idempotent $e'$ such that $\mu (e) = \mu (e')$ and $e' \leq f$.
By assumption, $e \, \mathscr{D} \, e'$.
It follows that $SeS \subseteq SfS$. 
\end{proof}

A natural example of a Boolean inverse monoid equipped with an invariant mean is the symmetric inverse monoid $I_{n}$.
Define $\mu (1_{A}) = \frac{\left| A \right|}{n}$.
In other words, we assign probability $\frac{1}{n}$ to each letter.
The proof of the following can easily be deduced from the case of finite symmetric inverse monoids
and the closure of the stated properties under finite direct products.

\begin{lemma}\label{lem:needful} 
Semisimple inverse monoids are completely semisimple and $\mathscr{D}$ preserves complementation.
\end{lemma}

Our first main theorem is the following.

\begin{theorem}\label{the:finite_case} 
Every finite MV-algebra can be co-ordinatized by a semisimple inverse monoid.
\end{theorem}
\begin{proof}
The core of the proof is the following.
Let $I_{n}$ be the finite symmetric inverse monoid on $n$ letters.
By Lemma~\ref{lem:needful}, such monoids are Foulis monoids.
It is well-known that they also satisfy the lattice condition, but we shall prove this explicity since the method we use is important for what we do later.
Define $\kappa \colon S/\mathscr{J} \rightarrow \mathbb{Z}$ by
$\kappa [1_{A}] = \left| A \right|$.
In other words, we map an idempotent to the cardinality of the set on which it is partial identity.
Observe that two idempotents are $\mathscr{D}$-related if, and only if, these cardinalities are the same.
It follows that $\kappa$ is a bijection from $I_{n}/ \mathscr{J}$ to the set $\mathbf{n} = \{0,1, \ldots, n\}$.
It is trivial that it is induces an order isomorphism when the set $\mathbf{n}$ carries the usual order.
It follows that the lattice condition is satisfied with the lattice operations being min and max.
The partial operation $\oplus$ translates into partial addition: if $r,s \in \mathbf{n}$ then $r \oplus s = r + s$ if $r + s \leq n$,
otherwise it is undefined.
The prime operation translates into $s' = n - s$.
We now describe the operation $\boxplus$.
By definition 
$$r \boxplus s = r + \mbox{min}(r',s).$$
We consider two cases.
Suppose first that $r + s \leq n$.
Then $s \leq n - r = r'$.
It follows that in this case $r \boxplus s = r + s$.
Next suppose that $r + s > n$.
Then $s > n - r = r'$.
It follows that in this case 
$r \boxplus s = n$.
We have therefore shown that $I_{n}/\mathscr{J}$ gives rise to the MV-algebra
known as the {\em \L ukasiewicz chain} $L_{n+1}$ \cite{M3}.

To prove the full theorem, we now use the result that every finite MV-algebra is a finite direct product of  \L ukasiewicz chains. 
See  \cite[Proposition 3.6.5]{CDM} or part~2 of \cite[Theorem~11.2.4]{M3}.
Such algebras can clearly be co-ordinatized by finite direct products of finite symmetric inverse monoids and so by semisimple inverse monoids.
\end{proof}

%%%%%%%%%%%%%%%%%%%%%%%%%%%%%%%%%%%%%%%%%%%%%%%%%%%%%%%%%%%%%%%%%%%%%%%%%%%%%%%%%%%%%%%%%%
\section{AF inverse monoids}

In this section, we shall define the class of  approximately finite (AF) inverse monoids and derive their basic properties.
The term {\em AF inverse semigroup} was also used in \cite{PZ} for inverse semigroups generated according to a quite complex recipe,
whereas in \cite{K}, Kumjian defines {\em AF localizations} which he states may be viewed `in  some sense' as inductive limits of finite localizations.
Our definition is simpler than either of the above definitions and shadows that of the definition of AF $C^{\ast}$-algebras.
It works because we use the correct definition of morphism between semisimple inverse monoids.
In any event, our AF monoids will turn out to be Foulis monoids, and they will provide one of the key ingredients in proving our main theorem.
Good sources for Bratteli diagrams and the construction of $C^{\ast}$-algebras from them are \cite{Effros,Goodearl}.

Fundamental to our work is the choice of the correct morphisms.
A {\em morphism} between Boolean inverse $\wedge$-monoids is a unital homomorphism 
that maps zeros to zeros, preserves all compatible binary joins and preserves all binary meets.
We define the {\em kernel} of a morphism to be all the elements that are mapped to zero.

\begin{lemma}\label{lem:kernel} A morphism $\theta \colon S \rightarrow T$ between Boolean inverse $\wedge$-monoids
is injective if, and only if, its kernel is zero.
\end{lemma}
\begin{proof} Only one direction needs proving.
Suppose that the kernel of $\theta$ is zero but there exist non-zero elements $a$ and $b$ such that $\theta (a) = \theta (b)$.
Since $\theta (a \wedge b) = \theta (a) \wedge \theta (b)$, we have that $\theta (a \wedge b) = \theta (a) = \theta (b)$.
Thus without loss of generality, we may as well assume that $b \leq a$.
We may construct an element $x$ such that $a = b \vee x$ and $b \wedge x = 0$;
define $x = a(a^{-1}a\overline{(b^{-1}b)})$.
Then $\theta (a) = \theta (b) \vee \theta (x)$ and $\theta (b) \wedge \theta (x) = 0$.
Since $\theta (a) = \theta (b)$ we deduce that $\theta (x) \leq \theta (a)$.
But then $\theta (x) = 0$ and so $x = 0$.
We deduce that $a = b$, as required.
\end{proof}

%%%%%%%%%%%%%%%%%%%%%%%%%%%%%%%%%%%%%%%%%%%%%%%%%%%%%%%%%%%%%%%%%%%%%%%%%%%%%%%%%%%%%%%%%%%%%%%%%%%%%%%%%%%
\subsection{The construction of AF monoids}

In the symmetric inverse monoid $I_{n}$, 
we denote by $e_{ij}$ the partial bijection with domain $\{j\}$ and codomain $\{i\}$.
The elements $e_{ii}$ are idempotents.
Every element of $I_{n}$ can be written as a unique orthogonal join of the elements $e_{ij}$.
In the case of idempotents, only the elements of the form $e_{ii}$ are needed.
Consider now the set of all $n \times n$ matrices whose entries are drawn from $\{0,1\}$
in which each row and each column contains at most one non-zero element.
The set of all such matrices is denoted by $R_{n}$ and called the set of {\em Rook matrices} \cite{Solomon}.
Fix an ordering of the set of letters of an $n$-element set.
For each $f \in I_{n}$ define $M(f)_{ij} = 1$ if $i = f(j)$ and $0$ otherwise.
In this way, we obtain a bijection between $I_{n}$ and $R_{n}$ which maps the identity function to the identity matrix
and which is a homomorphism between function composition and matrix multiplication. 
Thus the rook matrices $R_{n}$ provide isomorphic copies of $I_{n}$.
We have that $f \leq g$, the natural partial order, if and only if $M(f)_{ij} = 1 \Rightarrow M(g)_{ij} = 1$.
The meet $f \wedge g$ corresponds to the {\em freshman product}\footnote{That is, corresponding entries are multiplied.} of $M(f)$ and $M(g)$.
The elements $e_{ij}$ correspond to those rook matrices which are matrix units.
Let $A$ and $B$ be rook matrices of sizes $m \times m$ and $n \times n$, respectively.
We denote by $A \oplus B$ the $(m + n) \times (m + n)$ rook matrix
$$\left(\begin{array}{cc}
A & 0\\
0 & B
\end{array} \right)$$
We may iterate this construction.
We write $sA = A \oplus \ldots \oplus A$ where the sum has $s$ summands.
There is no ambiguity with scalar multiplication because such multiplication is not defined for rook matrices.
More generally, we can form sums such as $s_{1}A_{1} \oplus \ldots \oplus s_{m}A_{m}$. 
There are many isomorphisms between $I_{n}$ and $R_{n}$ but the only ones that we will
need are those determined by choosing a total ordering of the letters $\{1, \ldots, n\}$.
We shall call such isomorphisms {\em letter isomorphisms}.
We shall also be interested in isomorphisms from 
$I_{n_{1}} \times \ldots \times I_{n_{k}}$ 
to
$R_{n_{1}} \times \ldots \times R_{n_{k}}$ 
induced by letter isomorphisms from $I_{n_{i}}$ to $R_{n_{i}}$.
We shall also refer to these as letter isomorphisms.

Our first goal is to classify morphisms between semisimple inverse monoids.
We begin with a lemma that is a rare example of an arithmetic result in semigroup theory.

\begin{lemma}\label{lem:lagrange} There is a morphism from $I_{m}$ to $I_{n}$  if, and only if, $m \mid n$.
\end{lemma}
\begin{proof} Assume first that there is a morphism $\theta \colon I_{m} \rightarrow I_{n}$.
We may write $1 = \bigvee_{i=1}^{m} e_{ii}$ an orthogonal join.
Since $\theta$ is a morphism, we have that $\theta (1) = 1$ 
and 
$\theta \left( \bigvee_{i=1}^{m} e_{ii} \right) =  \bigvee_{i=1}^{m} \theta (e_{ii})$.
Thus
$1 =  \bigvee_{i=1}^{m} \theta (e_{ii})$.
Orthogonality is preserved by homomorphisms that map zeros to zeros.
Thus the union on the righthandside above is an orthogonal union.
Clearly $e_{ii} \, \mathscr{D} \, e_{jj}$ for all $i$ and $j$.
Thus $\theta (e_{ii}) \, \mathscr{D} \, \theta (e_{jj})$.
But two idempotents in a symmetric inverse monoid are $\mathscr{D}$-related
precisely when their domains of definition have the same cardinality.
Thus $\theta (e_{ii}) = 1_{A_{i}}$
where the sets $A_{1}, \ldots, A_{m}$ are pairwise disjoint and have the same cardinality $s$, say.
It follows that $n = sm$, and so $m \mid n$, as claimed.

To prove the converse, suppose that $n = sm$.
Choose letter isomorphisms from $I_{m}$ to $R_{m}$ and $I_{n}$ to $R_{n}$.
Define a map from $R_{m}$ to $R_{n}$ as follows $A \mapsto sA$.
It is easy to check that this is a morphism.
Thus we get a morphism from $I_{m}$ to $I_{n}$, as claimed.
\end{proof}

If $n = sm$, then the morphism from $R_{m}$ to $R_{n}$ defined by $A \mapsto sA$ is called a {\em standard morphism}.
Our next result says that, up to letter isomorphisms, all such morphisms are described by standard morphisms.

\begin{lemma}\label{lem:rigidity} Suppose that $m \mid n$ where $n = sm$.
Let $\theta \colon I_{m} \rightarrow I_{n}$ be a morphism and let $\alpha \colon I_{m} \rightarrow R_{m}$
be a letter isomorphism.
Then there is a standard map $\sigma \colon R_{m} \rightarrow R_{n}$
and a letter isomorphism $\beta \colon I_{n} \rightarrow R_{n}$ such that
$\theta = \beta^{-1}\sigma \alpha$.
In particular, every morphism from $I_{m}$ to $I_{n}$ is isomorphic to a standard map.
\end{lemma}
\begin{proof} Let $\theta \colon I_{m} \rightarrow I_{n}$ be a morphism.
We begin as in the proof of Lemma~\ref{lem:lagrange}.
Choose any ordering of the letters of $I_{m}$ and let $\alpha \colon I_{m} \rightarrow R_{m}$ be the corresponding isomorphism.
We may suppose that the letters are labelled $1, \ldots, n$.
Define the elements $e_{ij}$ relative to that ordering.
Let $1 = \bigvee_{i=1}^{m} e_{ii}$.
Then 
$1 =  \bigvee_{i=1}^{m} \theta (e_{ii})$
where $\theta (e_{ii}) = 1_{A_{i}}$
and the sets $A_{1}, \ldots, A_{m}$ are pairwise disjoint and have the same cardinality $s$.
Let $A_{i} = \{x_{i1}, \ldots, x_{is}\}$ where $i = 1, \ldots, m$.
Now order the elements of $\bigcup_{i=1}^{m} A_{i}$ as follows
$$x_{11}, x_{21}, \ldots, x_{m1}, \ldots, x_{1s}, \ldots, x_{ms}.$$
With this ordering, construct an isomorphism $\beta \colon I_{n} \rightarrow R_{n}$.
Let $\sigma \colon I_{m} \rightarrow I_{n}$ be the standard map $A \mapsto sA$.
We claim that $\theta = \beta^{-1}\sigma\alpha$.
It's enough to verify this for the partial bijections $e_{ij}$.
We have that 
$$e_{jj} \stackrel{e_{ij}}{\longrightarrow} e_{ii}.$$
Thus $\theta (e_{ij})$ has domain the domain of definition of $\theta (e_{jj})$ and image the image of definition of $\theta (e_{ii})$.
The domain of definition of $\theta (e_{jj})$ is the set $A_{j}$.
If the rook matrix of $e_{jj}$ is the matrix $M$ which has one non-zero entry in row $j$ and column $j$,
the matrix of $\theta (e_{jj})$ relative to the above ordering of letters will be $sM$.
The proof now readily follows.
\end{proof}

If $S$ is an inverse monoid and $e$ is any idempotent then $eSe$ is an inverse subsemigroup called a {\em local submonoid} (sic).
Consider now the symmetric inverse monoid $I_{n}$.
Then an idempotent $e = 1_{A}$ where $A \subseteq \{1, \ldots, n\}$.
It follows that the local submonoid $eI_{n}e$ is simply $I_{A}$, the symmetric inverse monoid on the set of letters $A$. 

Let $s_{1}m(1) + \ldots + s_{k}m(k) = n$, where the $s_{i}$ are non-negative integers.
Define the corresponding {\em standard morphism} $\sigma \colon R_{m(1)} \times \ldots \times R_{m(k)} \rightarrow R_{n}$ 
by $\sigma ((A_{1}, \ldots, A_{k})) = s_{1}A_{1} \oplus \ldots \oplus s_{k}A_{k}$
We may now classify morphisms from semisimple inverse monoids to symmetric inverse monoids.

\begin{lemma}\label{lem:lagrange_more} \mbox{}
\begin{enumerate}

\item There is a morphism from $I_{m(1)} \times \ldots \times I_{m(k)}$ to $I_{n}$ if, and only if,
there exist non-negative integers $s_{1}, \ldots, s_{k}$ such that  $s_{1}m(1)  + \ldots + s_{k}m(k) = n$.  

\item For each morphism $\theta \colon I_{m(1)} \times \ldots \times I_{m(k)} \rightarrow I_{n}$ 
and for each letter isomorphism  
$\alpha \colon I_{m(1)} \times \ldots \times I_{m(k)} \rightarrow R_{m(1)} \times \ldots \times R_{m(k)}$ 
there exists a letter isomorphism $\beta \colon I_{n} \rightarrow R_{n}$
and a standard morphism $\sigma \colon  R_{m(1)} \times \ldots \times R_{m(k)} \rightarrow R_{n}$ such that $\theta = \beta^{-1} \sigma \alpha$.

\end{enumerate}
\end{lemma}
\begin{proof} (1) Denote the set of letters of $I_{n}$ by $X$.
Put $S = I_{m(1)} \times \ldots \times I_{m(k)}$.
The identity of this monoid is the $k$-tuple of identities whose $i$th component is the identity of $I_{m(i)}$.
Define $\mathbf{e}_{i}$ to be the idempotent of $S$ all of whose elements are zero except the $i$th which is the identity of $I_{m(i)}$.
Then $1 = \bigvee_{i=1}^{k} \mathbf{e}_{i}$ is an orthogonal join.
Thus $1 = \bigvee_{i=1}^{k} \theta (\mathbf{e}_{i})$ is an orthogonal join and the identity function on $X$.
Let $\theta (\mathbf{e}_{i}) = 1_{X_{i}}$.
Denote the cardinality of $X_{i}$ by $a_{i}$.
The non-empty $X_{i}$ form a partition of $X$.
It follows that $n = a_{1} + \ldots + a_{k}$.
For each $i$, where $X_{i} \neq \emptyset$, we have that
$\theta(\mathbf{e}_{i})I_{n}\theta(\mathbf{e}_{i}) = I_{X_{i}}$, a symmetric inverse monoid on $a_{i}$ letters.
Now the morphism $\theta$ restricts to a morphism $\theta_{i}$ from $\mathbf{e}_{i}S\mathbf{e}_{i}$ to $\theta (\mathbf{e}_{i}) I_{n} \theta (\mathbf{e}_{i}) = I_{X_{i}}$.
But we have that $\mathbf{e}_{i}S\mathbf{e}_{i} \cong I_{m(i)}$.
We therefore have an induced morphism from $I_{m(i)}$ to $I_{a_{i}}$.
Thus by Lemma~\ref{lem:lagrange},  $a_{i} = s_{i}m(i)$ for some non-zero $s_{i}$
Hence $s_{1}m(1)  + \ldots + s_{k}m(k) = n$.
The converse is proved using a standard morphism defined as above.

(2) We continue with the notation introduced in part (1).
Let $\alpha = (\alpha_{1}, \ldots,\alpha_{k})$ 
be a letter isomorphism from
$I_{m(1)} \times \ldots \times I_{m(k)}$ 
to
$R_{m(1)} \times \ldots \times R_{m(k)}$.
In what follows, we need only deal with the $i$ where $X_{i} \neq \emptyset$.
Let $\iota_{i} \colon I_{m(i)} \rightarrow S$ be the obvious embedding.
Put $\theta_{i} = \theta \iota_{i}$.
Then $\theta_{i} \colon I_{m(i)} \rightarrow I_{X_{i}}$.
There is therefore a letter isomorphism $\beta_{i} \colon I_{X_{i}} \rightarrow R_{a_{i}}$ obtained through a specific ordering of the elements of $X_{i}$
and the standard map $\sigma_{i} \colon R_{m(i)} \rightarrow R_{a_{i}}$ given by $A \mapsto s_{i}A$
such that $\theta_{i} = \beta_{i}^{-1}\sigma_{i}\alpha_{i}$.
We order the letters of $I_{n}$ as $X_{1},\ldots,X_{k}$ with the ordering within each $X_{i}$ chosen as above.
Define $\beta \colon I_{n} \rightarrow R_{n}$ to be the corresponding letter isomorphism.
Then $\sigma = \sigma_{1} \oplus \ldots \oplus \sigma_{k}$. 
\end{proof}

\begin{remark}\label{rem:noodle} 
{\em Observe that
$$s_{i} = \frac{\left| \theta (\mathbf{e}_{i})\right|}{m(i)}$$
where $\left|   \theta (\mathbf{e}_{i}) \right|$ denotes the cardinality of the set on which the idempotent $\theta (\mathbf{e}_{i})$ is defined.}
\end{remark}

We suppose we are given
$I_{m(1)} \times \ldots \times I_{m(k)}$ 
and
$I_{n(1)} \times \ldots \times I_{n(l)}$.
Put 
$\mathbf{m} = (m(1) \ldots m(k))^{T}$
and
$\mathbf{n} = (n(1) \ldots n(l))^{T}$.
Assume that we are given an $l \times k$ matrix $M$,  where $M_{ij} = s_{ij}$,
non-negative natural numbers,
such that $M \mathbf{m} = \mathbf{n}$.
Then we define a standard map 
$\sigma$
from 
$R_{m(1)} \times \ldots \times R_{m(k)}$ 
to
$R_{n(1)} \times \ldots \times R_{n(l)}$
by
$$
\left(\begin{array}{c}
A_{1}\\
\ldots\\
A_{k}
\end{array}
\right)
\mapsto
M
\left(\begin{array}{c}
A_{1}\\
\ldots\\
A_{k}
\end{array}
\right)
$$

\begin{proposition}\label{prop:classifying_morphisms}
Given a morphism 
$\theta \colon S = I_{m(1)} \times \ldots \times I_{m(k)} \rightarrow  I_{n(1)} \times \ldots \times I_{n(l)} = T$
and a letter isomorphism $\alpha \colon I_{m(1)} \times \ldots \times I_{m(k)} \rightarrow R_{m(1)} \times \ldots \times R_{m(k)}$
there is a letter isomorphism $\beta \colon I_{n(1)} \times \ldots \times I_{n(l)} \rightarrow R_{n(1)} \times \ldots \times R_{n(l)}$
and a standard map $\sigma \colon R_{m(1)} \times \ldots \times R_{m(k)} \rightarrow  R_{n(1)} \times \ldots \times R_{n(l)}$
such that $\theta = \beta^{-1} \sigma \alpha$.
\end{proposition}
\begin{proof} We use the $l$ projection morphisms from $T$ to each of $I_{n(1)}, \ldots, I_{n(l)}$
composed with $\theta$ to get  morphisms from $S$ to each of $I_{n(1)}, \ldots, I_{n(l)}$ in turn.
We now apply Lemma~\ref{lem:lagrange_more}.
The separate results can now easily be combined to prove the claim.
\end{proof}

The data involved in describing a morphism 
from 
$I_{m(1)} \times \ldots \times I_{m(k)}$ 
to 
$I_{n(1)} \times \ldots \times I_{n(l)}$
can be encoded by means of a directed graph which we shall call a {\em diagram}.
We draw $k$ vertices, labelled $m(1)$ \ldots $m(k)$, in a line, the {\em upper vertices},
and then
we draw $l$ vertices, labelled $n(1) \ldots n(l)$, on the line below, the {\em lower vertices}.
We join the vertex labelled $m(j)$ to the vertex labelled $n(i)$ by means of $s_{ij}$ directed edges.
We require such graphs to satisfy the arithmetic conditions
$n(i) = s_{i1}m(1) + \ldots + s_{ik}m(k)$.
We call these the {\em combinatorial conditions}.
In other words, the matrix $M$ defined above is the adjacency matrix
where the upper vertices label the columns and the lower vertices label the rows.

\begin{remark} 
{\em In a diagram, each lower vertex is the target of at least one edge.
This is immediate by Lemma~\ref{lem:lagrange_more}.} 
\end{remark}

\begin{lemma}\label{lem:injectivity} Let $\sigma \colon S = I_{m(1)} \times \ldots \times I_{m(k)}
\rightarrow
I_{n(1)} \times \ldots \times I_{n(l)} = T$
be a standard map.
Then $\sigma$ is injective if, and only if, every upper vertex is the source of some directed edge.
\end{lemma}
\begin{proof} Without loss of generality, suppose that the upper vertex $m(1)$ is not the source of any edge.
Then all the elements $I_{m(1)} \times \{0\} \ldots \times \{0\}$
are in the kernel of $\sigma$ and so, in particular, $\sigma$ is not injective.
Now suppose that every upper vertex is the source of some edge.
Then clearly $\sigma$ has kernel equal to zero.
We now use Lemma~\ref{lem:kernel} to deduce that $\sigma$ is injective.
\end{proof}

We now recall a standard definition \cite{BR}.
A {\em Bratteli diagram} is an infinite directed graph $B = (V,E)$
with vertex-set $V$ and edge-set $E$ 
such that 
$V = \bigcup_{i=0}^{\infty} V(i)$
and 
$E = \bigcup_{i=1}^{\infty} E(i)$
are partitions of the respective sets into finite blocks, in the case of the vertices called {\em levels}, such that
\begin{enumerate}

\item $V(0)$ consists of one vertex $v_{0}$ we call the {\em root}.

\item Edges are only defined from $V(i)$ to $V(i+1)$, that is {\em adjacent levels}, and there are only finitely many edges from one level to the next.

\item Each vertex is the source of an edge and each vertex, apart from the root, is the target of an edge.
 \end{enumerate}

\begin{remark} {\em We have proved that each injective morphism between two semisimple inverse monoids
determines a diagram that satisfies the condition to be adjacent levels in a Bratteli diagram.}
\end{remark}

Let $B$ be a Bratteli diagram.
For each vertex $v$ we define its {\em size} $s_{v}$ to be the number of directed paths from the root $v_{0}$ in $B$ to $v$.
We now associate a semisimple inverse monoid with each level of the Bratteli diagram.
With the root vertex, we associate $S_{0} = I_{1}$, the two-element Boolean inverse $\wedge$-monoid.
With level $i \geq 2$, we associate the inverse monoid $S_{i}$.
This is constructed as follows.
List the $k$ vertices of level $i$ and then their respective sizes as $m(1), \ldots, m(k)$.
We put $S_{i} = I_{m(1)} \times \ldots \times I_{m(k)}$.
We now show how to define a morphism from $S_{i}$ to $S_{i+1}$.
List the $l$ vertices of level $i+1$ and then their respective weights as $n(1), \ldots, n(l)$.
In the Bratteli diagram, the vertex $m(j)$ will be joined to the vertex $n(i)$ by $s_{ij}$ edges.
The following is proved using a simple counting argument.

\begin{lemma} 
Adjacent levels of a Bratteli diagram satisfy the combinatorial conditions.
\end{lemma} 

It follows that we may define a standard morphism $\sigma_{i}$ from $S_{i}$ to $S_{i+1}$.
This will be injective by Lemma~\ref{lem:injectivity}.
We have therefore constructed a sequence of injective morphisms between semisimple inverse monoids
$$S_{0} \stackrel{\sigma_{0}}{\rightarrow} S_{1} \stackrel{\sigma_{1}}{\rightarrow} S_{2} \stackrel{\sigma_{2}}{\rightarrow} \ldots $$

%%%%%%%%%%%%%%%%%%%%%%%%%%%%%%%%%%%%%%%%%%%%%%%%%%%%%%%%%%%%%%%%%%%%%%%%%%%%%%%%%%%%%%%%%%%%%
We shall now describe direct limits of Boolean inverse monoids.
We begin with a well-known construction in semigroup theory.
Let 
$$S_{0} \stackrel{\tau_{0}}{\rightarrow} S_{1} \stackrel{\tau_{1}}{\rightarrow} S_{2} \stackrel{\tau_{2}}{\rightarrow} \ldots $$
be a sequence of inverse monoids and injective morphisms. {\em We use the dual order on $\mathbb{N}$. 
If $i,j \in \mathbb{N}$ denote by $i \wedge j$ the maximum element in of $\{i,j\}$.}
For $j < i$ define $\tau^{i}_{j} = \tau_{j-1} \ldots \tau_{i}$.
Thus $\tau^{i}_{i+1} = \tau_{i}$.
Define $\tau^{i}_{i}$ to be the identity function on $S_{i}$.
Clearly, if $k \leq j \leq i$ then  $\tau^{i}_{k} = \tau^{j}_{k}\tau^{i}_{j}$.
Put $S = \bigsqcup_{i=0}^{\infty} S_{i}$, a disjoint union of sets.
Let $a, b \in S$ where $a \in S_{i}$ and $b \in S_{j}$.
Define 
$$a \cdot b  = \tau^{i}_{i \wedge j}(a) \tau^{j}_{i \wedge j}(b).$$
Then $(S,\cdot)$ is a semigroup.
We shall usually represent multiplication by concatenation.
Observe that the set of idempotents of $S$ is the union of 
the set of idempotents of each of the $S_{i}$.
It is routine that idempotents commute.
In addition, $S$ is regular.
It follows that $S$ is an inverse semigroup.
The inverse of $a \in S$ where $a \in S_{i}$ is simply its inverse in $S_{i}$.
The identity element of $S_{0}$ is the identity for the semigroup $S$.
The monoid $S$ is said to be an {\em $\omega$-chain of inverse monoids}.

\begin{remark}
{\em The semigroup $S$ does not have a zero.
Instead, the set of zeros from each $S_{i}$ forms an ideal $\mathscr{Z}$ in $S$.
If we form the quotient monoid, $S/\mathscr{Z}$ then essentially all the elements of $S \setminus \mathscr{Z}$ remain the same
whereas the elements of $\mathscr{Z}$ are rolled up into one zero.}
\end{remark}

Denote the identity of $S_{i}$ by $e_{i}$.
Put $\mathscr{E} = \{e_{i} \colon i \in \mathbb{N} \}$.
Then $\mathscr{E}$ forms a subsemigroup of the semigroup $S$ and is a subset of the centralizer of $S$.
For each $a \in S$, there exists $e \in \mathscr{E}$ such that $a = ea = ae$.
Define $a \equiv b$ if, and only if, $ae = ae$ for some $e \in \mathscr{E}$.
Then $\equiv$ is a congruence on $S$ and the quotient is an inverse monoid with zero.
%%%%%%%%%%%%%%%%%%%%%%%%%%%%%%%%%%%%%%%%%%%%%%%%%%%%%%%%%%%%%%%%%%%%%%%%%%%%%%%%%%

\begin{lemma} Let 
$$S_{0} \stackrel{\tau_{0}}{\rightarrow} S_{1} \stackrel{\tau_{1}}{\rightarrow} S_{2} \stackrel{\tau_{2}}{\rightarrow} \ldots $$
be a sequence of Boolean inverse $\wedge$-monoids and injective morphisms.
Then the direct limit $\varinjlim S_{i}$   is a Boolean inverse $\wedge$-monoid.
In addition, we have the following.
\begin{enumerate}

%1
\item If all the $S_{i}$ are fundamental then  $\varinjlim S_{i}$ is fundamental.

%2
\item If all the $S_{i}$ are factorizable then  $\varinjlim S_{i}$ is factorizable.

%3
\item If all the $S_{i}$ are completely semisimple then  $\varinjlim S_{i}$ is completely semisimple.

%4
\item If all the $S_{i}$ have the property that $\mathscr{D}$ preserves complementation then  $\varinjlim S_{i}$ satisfies the property that $\mathscr{D}$ preserves complementation.

%5
\item The group of units of  $\varinjlim S_{i}$ is the direct limit of the groups of units of the $S_{i}$.

\end{enumerate}
\end{lemma}
\begin{proof} We construct $\omega$-chain of inverse monoids $S$, as above.
Let $j \leq i$ and let $b \in S_{j}$ and $a \in S_{i}$.
Then $b = \tau^{i}_{j}(a)$ if, and only if,  $b = a \cdot e_{j}$
It follows, in particular, that $b \leq a$.
Let $a \in S_{i}$ and $b \in S_{j}$.
Then there is $l \leq i,j$ 
such that $\tau^{i}_{l} (a) = \tau^{j}_{l} (b)$
if, and only if, $e_{l}a = e_{l}b$.
Define $a \equiv b$ if, and only if, there exists $e \in \mathscr{E}$ such that
$ea = eb$.
Then, as above, $\equiv$ is a congruence on the inverse semigroup $S$.
It is idempotent-pure because the $\tau_{i}$ are injective.
We denote the $\equiv$-class containing the element $a$ by $[a]$.
We denote the set of $\equiv$-classes by $S_{\infty}$.
All the elements in $\mathscr{Z}$ are identified and so $S_{\infty}$ is an inverse monoid with zero.
Observe that the product is given by
$$[a][b] = [\tau^{i}_{i \wedge j}(a) \tau^{j}_{i \wedge j}(b)].$$

Let $[a],[b] \in S_{\infty}$ where $a \in S_{i}$ and $b \in S_{j}$.
Then $[a] \sim [b]$ if and only if $\tau^{i}_{i \wedge j}(a) \sim \tau^{j}_{i \wedge j}(b)$.
It is now routine to check that $S_{\infty}$ has binary compatible joins,
and that multiplication distributes over such joins.
Let $[a],[b] \in S_{\infty}$ where $a \in S_{i}$ and $b \in S_{j}$.
Put $c = \tau^{i}_{i \wedge j}(a) \wedge \tau^{j}_{i \wedge j}(b)$.
We show that $[c] = [a] \wedge [b]$.
Observe that if $x,y \in S_{l}$ and $x \leq y$ then $[x] \leq [y]$.
We have that $[a] = [\tau^{i}_{i \wedge j}(a)]$ and $[b] = [\tau^{j}_{i \wedge j}(b)]$.
Clearly $[c] \leq [a],[b]$.
It is now routine to check that if $[d] \leq [a],[b]$ then $[d] \leq [c]$.
If $[e]$ is an idempotent then the operation $\overline{[e]} = [\overline{e}]$ is well-defined
and $[e] \wedge \overline{[e]} = [0]$ and $[e] \vee \overline{[e]} = [1]$.
We have therefore shown that $S_{\infty}$ is a Boolean inverse $\wedge$-monoid.

Define $\phi_{i} \colon S_{i} \rightarrow S_{\infty}$ by $s \mapsto [s]$.
This map is evidently a morphism and whenever $j \leq i$ we have that
$\phi_{j} \tau^{i}_{j} = \phi_{i}$.
Now let $T$ be a Boolean inverse $\wedge$-monoid such that there are morphisms
$\theta_{i} \colon S_{i} \rightarrow T$ such that whenever $j \leq i$ we have that
$\theta_{j} \tau^{i}_{j} = \theta_{i}$.
Define $\psi \colon S_{\infty} \rightarrow T$ by $\psi ([a]) = \theta_{i} (a)$ if $a \in S_{i}$.
That this is a well-defined morphism witnessing that $S_{\infty}$ is indeed the direct limit is
now routine.

(1) The proof of this is straightforward. In particular, it uses the fact that if the image of an element under an 
injective morphism is an idempotent then that element is an idempotent.
(2) If $[a]$ is an arbitrary element where $a \in S_{i}$. 
Then $a \leq g$ where $g$ is invertible in $S_{i}$ and so $[a] \leq [g]$ in $S_{\infty}$.
But if $g$ is invertible then $[g]$ is invertible. 
(3) Straightforward.
(4) Straightforward.
(5) This follows from the fact that, since the morphisms are all injective, the element $[a]$ is invertible if and only if $a$ is invertible.
\end{proof}

It follows that with each Bratteli diagram $B$ we may associate a Boolean inverse $\wedge$-monoid constructed as a direct limit
of semisimple inverse monoids and standard morphisms.
We denote this inverse monoid by $\mathsf{I}(B)$.

\begin{lemma} Let 
$$S_{0} \stackrel{\tau_{0}}{\rightarrow} S_{1} \stackrel{\tau_{1}}{\rightarrow} S_{2} \stackrel{\tau_{2}}{\rightarrow} \ldots $$
be a sequence of semisimple inverse monoids and injective morphisms.
Then the direct limit $\varinjlim S_{i}$  
is isomorphic to $I(B)$ for some Bratteli diagram $B$.
\end{lemma}
\begin{proof} 
This follows by repeated application of Proposition~\ref{prop:classifying_morphisms}.
\end{proof}

We call any inverse monoid constructed in this fashion an {\em AF inverse monoid}.
We may now summarize what we have found in this section in the following theorem.

\begin{theorem}\label{the:holiday} AF inverse monoids are fundamental Foulis $\wedge$-monoids.
Their groups of units are direct limits of finite direct products of finite symmetric groups where the morphisms between 
successive such direct products are by means of diagonal embeddings.
\end{theorem}

The groups of units of AF inverse monoids are therefore the groups studied in \cite{DM,KS,LN}.

%%%%%%%%%%%%%%%%%%%%%%%%%%%%%%%%%%%%%%%%%%%%%%%%%%%%%%%%%%%%%%%%%%%%%%%%%%%%%%%%%%%%%%%%%%%%%%%%
\subsection{An example}

In this section, we shall construct a concrete example of an infinite MV-algebra that can be co-ordinatized
by an inverse monoid.
As we shall see, the monoid we construct is an analogue of the CAR algebra \cite{M2}.
Recall that a non-negative rational number is said to be {\em dyadic} if it can be written in the form $\frac{a}{2^{b}}$
for some natural numbers $a$ and $b$.
The goal of the remainder of this section is to prove the following.

\begin{theorem}\label{them:dyadic} The MV-algebra of dyadic rationals in the closed unit interval $[0,1]$
can be co-ordinatized by an inverse monoid.
\end{theorem}

The inverse monoid in question will be what we term the dyadic inverse monoid.
This will be constructed as a submonoid of the Cuntz inverse monoid which we describe first.

%%%%%%%%%%%%%%%%%%%%%%%%%%%%%%%%%%%%%%%%%%%%%%%%%%%%%%%%%%%%%%%%%%%%%%%%%%%%%%%%%%%%%%%%%%%%%%%%
\begin{center}
{\bf String theory}
\end{center}

As a first step, we construct an inverse monoid, $C_{n}$, called the Cuntz inverse monoid.
This was first described in \cite{Law2007a,Law2007b} but we have improved on the presentation given there and so we give it  in some detail.

We begin by describing how we shall handle the Cantor space and its clopen subsets.
Let $A$ be a finite alphabet with $n$ elements where $n \geq 2$.
We shall primarily be interested in the case where $A = \{a,b\}$.
We denote by $A^{\ast}$ the set of all finite strings over $A$.
The empty string is denoted by $\varepsilon$.
We denote the total number of symbols occurring in the string $x$, counting repeats, by $\left| x \right|$.
This is called the {\em length} of $x$.
If $x,y \in A^{\ast}$ such that $x = yu$ for some finite string $u$, then we say that $y$ is a {\em prefix} of $x$.
We define $x \preceq y$ if and only if $x = yu$.
This is a partial order on $A^{\ast}$ called the {\em prefix order}.
Observe that if $x \preceq y$ then $x$ is at least as long as $y$.
A pair of strings $x$ and $y$ are said to be {\em prefix comparable} if $x \preceq y$ or $y \preceq x$.
A subset $X \subseteq A^{\ast}$ is called a {\em prefix subset} if for all $x,y \in X$ we have that $x \preceq y$ implies that $x = y$.
If $X$ is a prefix subset and contains the empty string then it contains only the empty string.
If $X$ is a prefix subset such that whenever $X \subseteq Y$, where $Y$ is a prefix subset, we have that $X = Y$,
then $X$ is called a {\em maximal prefix subset}.
Prefix subsets are often called {\em prefix codes}.
We shall only consider {\em finite} prefix sets in this paper.
If $X \subseteq A^{\ast}$ is a finite set, define $\mbox{max}(X)$ to be the maximal elements of $X$ under the prefix ordering.
It is immediate that $\mbox{max}(X)$ is a prefix set.
We define the {\em length} of a prefix set $X$ to be the maximum length of the strings belonging to $X$.
We say that a prefix set $X$ is {\em uniform of length $l$} if all strings in $X$ have length $l$.

By $A^{\omega}$ we mean the set of all right-infinite strings over $A$.
The set $A^{\omega}$ is equipped with the topology inherited from its representation as the space $A^{\mathbb{N}}$,
where $A$ is given the discrete topology.
It is the {\em Cantor space}.
Up to homeomorphism, it is independent of the cardinality of $A$.
Its clopen subsets are those subsets of the form $XA^{\omega}$ where $X \subseteq A^{\ast}$ is a finite set.
The following result is well-known.

\begin{lemma}\label{lem:prefix_comparability}  $xA^{\omega} \cap yA^{\omega} \neq \emptyset$
if and only if  $x$ and $y$ are prefix comparable.
If $x$ and $y$ are prefix-comparable, then 
either $xA^{\omega} \subseteq yA^{\omega}$ or $yA^{\omega} \subseteq xA^{\omega}$.
In particular, if $x \preceq y$ then $xA^{\omega} \subseteq yA^{\omega}$.
\end{lemma}
\begin{proof} Suppose that $xA^{\omega} \cap yA^{\omega} \neq \emptyset$.
Let $w \in xA^{\omega} \cap yA^{\omega}$.
Then $w = xu = yv$ where $u$ and $v$ are infinite strings.
If $x$ and $y$ have the same length, then $x =y$.
Otherwise we may assume, without loss of generality, that $\left| x \right| > \left| y \right|$.
It follows that $y$ is a prefix of $x$ and we can write $x = yc$ for some finite string $c$.
Clearly,  $xA^{\omega} \subseteq yA^{\omega}$.
\end{proof}

It follows by the above lemma, that if  $U = XA^{\omega}$ is a clopen set for some finite set  $X$, 
then $U = \mbox{max}(X)A^{\omega}$.
Thus we may choose the set $X$ to be a prefix set.
This we shall always do from now on.
If $U$ is a clopen subset and $U = XA^{\omega}$, where $X$ is a prefix set, then we say that $X$ is a {\em generating set} of $U$.
Observe that if
$U = XA^{\omega}$ where $X$ is a prefix set,
then
$$U = \bigcup_{i=1}^{m} x_{i}A^{\omega}$$
is actually a disjoint union.
The clopen subsets form a basis for the topology on the Cantor space.
The sets  $XA^{\omega}$ are called {\em cylinder sets}.
Finite sets $X$ will often be represented using the notation of {\em regular languages}.
Thus if $X = \{x_{1}, \ldots, x_{m} \}$, we shall also write $X = x_{1} + \ldots + x_{m}$.
For more on infinite strings and proofs of any of the claims above, see \cite{PP}.

\begin{example} {\em Let $A = \{a,b\}$.
The representation of clopen subsets by prefix sets is not unique.
For example, $aA^{\omega} = (aa + ab)A^{\omega}$,
and
$A^{\omega} = (a + b)A^{\omega}$.}
\end{example}

The lack of uniqueness in the use of prefix sets to describe clopen subsets is something we shall have to handle.
The next few results provide the means for doing so.
We make no claims for originality, but include these results for the sake of clarity.

\begin{lemma}\label{lem:maximal} Let $A$ be a finite alphabet and let $X$ be a prefix set over $A$.
Then $XA^{\omega} = X$ if, and only if, $X$ is a maximal prefix set.
\end{lemma} 
\begin{proof} Suppose first that $X$ is a maximal prefix set of length $l$.
Let $w$ be any infinite string.
Write $w = uw'$ where $w'$ is infinite and $u$ is the prefix of $w$ of length $l$.
The set $X + u$ properly contains $X$ and so cannot be a prefix set.
Thus $u$ is prefix-comparable with an element of $X$.
But, because of its length, it either equals an element of $X$ or an element of $X$ is a proper prefix of $u$.
Thus there exists $x \in X$ such that $u = xu'$.
It follows that $w = xu'w'$ and so $w \in XA^{\omega}$.
 
Conversely, suppose that $XA^{\omega} = X$.
We prove that $X$ is a maximal prefix set.
Suppose not.
Then there is at least one finite string $u$ such that $X + u$ is a prefix set.
Let $w$ be any infinite string.
Clearly, $uw \in A^{\omega}$.
But then $uw = xw'$ where $x \in X$.
Thus $uA^{\omega} \cap xA^{\omega} \neq \emptyset$.
By Lemma~\ref{lem:prefix_comparability},
it follows that  $u$ and $x$ are prefix-comparable, which is a contradiction.
\end{proof}

We shall now describe two operations on a prefix set.
In what follows, observe that for any $r \geq 0$, the set $A^{r}$ is a maximal prefix set.
The cases of interest below will always require $r \geq 1$.
Let $X$ be a prefix set.
Let $u \in X$.
Define
$$X^{+} = (X - u) + uA^{r},$$ 
where $r \geq 1$. We call $X^{+}$  an {\em extension} of $X$.
Let $u \in X$ such that $uA^{r} \subseteq X$ for some $r \geq 1$.
Define 
$$X^{-} = (X - uA^{r}) + u.$$ 
We call $X^{-}$ a {\em reduction} of $X$.
The proof of the following is straightforward.

\begin{lemma}\label{lem:ops} Let $X$ be a prefix set.
Then both $X^{+}$ and $X^{-}$ are prefix sets and
$XA^{\omega} = X^{+}A^{\omega} = X^{-}A^{\omega}$.
In addition,
$$X^{+-} = X \mbox{ and } X^{-+} = X.$$
\end{lemma}

Our next result shows that we may always replace a generating set by a uniform generating set.

\begin{lemma}\label{lem:uniformatizing} Let $U = XA^{\omega}$ where $X$ has length $l$.
Then for each $r \geq l$ we may find a prefix set $Y$ uniform of length $r$ such that $U = YA^{\ast}$.
\end{lemma}
\begin{proof} Let $U = XA^{\omega}$ where $X$ is a prefix set of length $l$.
If all the strings in $X$ have length $l$ then we are done.
Otherwise, let $u \in X$ such that $m = \left| u \right| < l$.
Then by Lemma~\ref{lem:ops}, we have that 
$X^{+} = (X - u) + uA^{l-m}$ is a prefix set and that $XA^{\omega} = X^{+}A^{\omega}$.
Thus the single string $u$ has been replaced by $\left| A \right|$ strings each of length $l$.
If all strings in $X^{+}$ have length $l$ we are done, else we repeat the above procedure.
In this way, we construct a prefix set $X'$ uniform of length $l$ such that $U = X'A^{\omega}$.
It is now clear how this process can be repeated to obtain prefix sets generating $U$ and uniform of any desired length $r \geq l$.
 \end{proof}

\begin{example} 
{\em Let $A = a + b$.
Consider the clopen set
$(aa + aba + b)A^{\omega}$.
The length of $aa + aba + b$ is 3.
Replace $b$ by $b(a + b)^{2}$ and replace $aa$ by $aa(a + b)$.
We therefore get the prefix set
$$aa(a + b) + aba + b(a + b)^{2}$$
and we have, in addition, that
$$(aa + aba + b)A^{\omega} = (aa(a + b) + aba + b(a + b)^{2})A^{\omega}.$$
}
\end{example}

Our next goal is to show that every clopen set has a `smallest' generating set, in a suitable sense.

\begin{lemma}\label{lem:piglet} If $XA^{\omega} \subseteq YA^{\omega}$, where $X$ and $Y$ are prefix sets,
then each element of $X$ is a prefix comparable with an element of $Y$.
\end{lemma} 
\begin{proof} Let $x \in X$.
Then $xA^{\omega} \subseteq XA^{\omega}$.
It follows that $xA^{\omega} = xA^{\omega} \cap YA^{\omega}$.
Thus $xA^{\omega} = \bigcup_{y \in Y} xA^{\omega} \cap yA^{\omega}$.
For at least one $y \in Y$, we must have that $xA^{\omega} \cap yA^{\omega} \neq \emptyset$.
By Lemma~\ref{lem:prefix_comparability}, it follows that $x$ and $y$ are prefix-comparable.  
\end{proof}

The following is immediate by the above lemma.

\begin{corollary}\label{cor:aslan} 
Let $XA^{\omega} = YA^{\omega}$ where $X$ and $Y$ are prefix sets both uniform of the same length.
Then $X = Y$.
\end{corollary}

We define the {\em weight} of a prefix set $X$ to be the sum $\sum_{x \in X} \left| x \right|$.

\begin{lemma}\label{lem:pooh} Let $U = XA^{\omega} = YA^{\omega}$ where $X$ and $Y$ have the same weight $p$.
Suppose, in addition, that any generating set of $U$ has weight at least $p$.
Then $X = Y$.
\end{lemma}
\begin{proof} Let $x \in X$.
By Lemma~\ref{lem:piglet}, there exists $y \in Y$ such that $x$ and $y$ are prefix-comparable.
Suppose that $x \neq y$.
Then, without loss of generality, we may assume that $x$ is a proper prefix of $y$.
Thus $y = xu$ for some finite string $u$.
Consider the set $(Y - y) + x$.
Observe that $U = ((Y - y) + x)A^{\omega}$.
It is not possible for any element of $Y - y$ to be a prefix of $x$ because then it would be a prefix of $y$ which is a contradiction.
It may happen that $x$ is a prefix of some elements of $Y - y$.
So we consider $Y' = \mbox{\rm max}(Y - y)$.
We have that $Y'$ is a generating set of $U$ and its weight is strictly less than $p$.
This is a contradiction.
We have therefore shown that if $x \in X$ then $x \in Y$.
By symmetry, we deduce that $X = Y$.
\end{proof}

\begin{lemma}\label{lem:kanga} Let $xA^{\omega} \subseteq YA^{\omega}$ where $Y$ is a prefix set.
Suppose that there is $y \in Y$ such that $y = xu$.
then
\begin{enumerate}

\item If $y' \in Y$ then either $y'A^{\omega} \cap xA^{\omega} = \emptyset$ or $y' = xv$ for some $v$.

\item Denote by $\overline{Y}$ the set of all elements of $Y$ that have $x$ as a prefix.
Then $x^{-1}\overline{Y}$ is a maximal prefix set.

\end{enumerate}
\end{lemma}
\begin{proof} (1) Suppose that $xA^{\omega} \cap xA^{\omega} \neq \emptyset$ where $y' \neq y$.
If $x = y'v$ then $y$ and $y'$ are prefix-comparable, which is a contradiction.
It follows that $y' = xv$.

(2) Suppose that  $x^{-1}\overline{Y}$ is not a maximal prefix set.
Let $z$ be a string that is not prefix comparable with any string in $x^{-1}\overline{Y}$.
Then $xz$ is not prefix-comparable with any element of $\overline{Y}$.
However, $xzA^{\omega} \subseteq xA^{\omega}$.
Thus $xz$ is prefix comparable with some element $y''$ of $Y$. 
If $y'' = xzz'$ for some $z'$ then $y'' \in \overline{Y}$, which is a contradiction.
Thus $xz = y''z'$.
Thus $x$ and $y''$ are prefix-comparable.
By (1) above, we must have that $y'' = xd$ for some string $d$.
Thus $d \in x^{-1}\overline{Y}$.
But $xz = xdz'$.
Thus $z = dz'$.
But this is a contradiction.
\end{proof}

\begin{lemma}\label{lem:tigger} Let $X$ be a prefix set.
Suppose that $xZ \subseteq X$ where $Z$ is a maximal prefix set, where $Z \neq \emptyset$.
Then it is possible to apply reduction to $X$.
\end{lemma}
\begin{proof} It is enough to show that we may apply reduction to $Z$.
Let $z \in Z$ be a string of maximal length.
Suppose that $z = z'a$ where $a \in A$.
I claim that $z'A \subseteq Z$.
Let $b \in A$ where $b \neq a$.
Then $z'b$ is a string the same length as $z$.
So it too has maximal length.
Since $Z$ is a maximal prefix set, it follows that $z'b$ must be prefix-comparable with some element of $Z$.
So it either belongs to $Z$, and we are done,
or some element of $Z$ of length at least one less is a prefix of $z'b$, which is impossible.
\end{proof}

\begin{proposition}\label{prop:wol} Let $U = XA^{\omega}$.
Construct the prefix code $X'$ from $X$ by carrying out any sequence of reductions until this is no longer possible.
Then $X'$ is a generating set of $U$ of minimum weight.
\end{proposition}
\begin{proof} The fact that $X'$ is a generating set follows by Lemma~\ref{lem:ops}.
Suppose that $U = YA^{\omega}$ where $Y$ has strictly smaller weight than $X$.
Let $x \in X'$.
Then by Lemma~\ref{lem:piglet}, $x$ must be prefix-comparable with some $y \in Y$.
Suppose for each $x \in X'$, it were the case that there was an element $y_{x} \in Y$ such that $x$ was a prefix of $y_{x}$.
If $x,x' \in X'$ were both prefixes of $y$, then they would have to be prefix-comparable.
It would then  follow that the weight of $Y$ was equal to or greater than the weight of $X'$, which is a contradiction.
Since the weights of the two prefix sets are different the sets cannot be equal.
It follows that there is at least one $x \in X'$ and $y \in Y$ such that $x = yu$ for some finite string $u$ of length $r \geq 1$.
We have that $yA^{\omega} \subseteq X'A^{\omega}$.
By Lemma~\ref{lem:kanga}, it follows that all the elements of $X'$ that have $y$ as a prefix
forms a subset $yZ$ where $Z$ is a maximal prefix code.
Then by Lemma~\ref{lem:tigger}, it is possible to apply a reduction to $X'$, which is a contradiction.
\end{proof}

By Proposition~\ref{prop:wol} and Lemma~\ref{lem:pooh}, 
it follows that every clopen subset is generated by a unique prefix set of minimum weight.
We call this the {\em minimum generating set}.

\begin{lemma}\label{lem:roo} Suppose that $U = XA^{\omega} = YA^{\omega}$.
Then $Y$ is obtained from $X$ by a finite sequence of extensions and reductions. 
\end{lemma}
\begin{proof} By means of a sequence of reductions $X$ may be converted to the minimum generating set $X_{U}$ by Lemma~\ref{prop:wol}.
Likewise $Y$ may be converted to the minimum generating set $X_{U}$.
Starting with $X_{U}$ we may therefore construct $Y$ by a sequence of extensions applying Lemma~\ref{lem:ops}.
Combing these two sequences together we may convert $X$ to $Y$.
\end{proof}

%%%%%%%%%%%%%%%%%%%%%%%%%%%%%%%%%%%%%%%%%%%%%%%%%%%%%%%%%%%%%%%%%%%%%%%%%%%%%%%%%%%%%%%%%%%
\begin{center}
{\bf The Cuntz inverse monoid}
\end{center}

We can now set about constructing an inverse monoid.
Let $A = a_{1} + \ldots + a_{n}$, though in the case $n = 2$, we shall usually assume that $A = a + b$.
The {\em polycyclic monoid $P_{n}$}, where $n \geq 2$,
is defined as a monoid with zero by the following presentation
$$P_{n} = \langle a_{1}, \ldots, a_{n}, a_{1}^{-1}, \ldots, a_{n}^{-1}
\colon \: a_{i}^{-1}a_{i} = 1 \,\mbox{and}\, a_{i}^{-1}a_{j} = 0, i \neq j \rangle.$$
It is, in fact, an inverse monoid with zero.
Every non-zero element of $P_{n}$ is of the form $yx^{-1}$ where
$x,y \in A_{n}^{\ast}$, and where we identify the identity with the
element $1 = \varepsilon \varepsilon^{-1}$.
The product of two elements $yx^{-1}$ and $vu^{-1}$ is zero unless
$x$ and $v$ are prefix-comparable.
If they are prefix-comparable then
$$yx^{-1} \cdot vu^{-1} = \left\{
\begin{array}{ll}
yzu^{-1}   & \mbox{if $v = xz$ for some string $z$}\\
y(uz)^{-1} & \mbox{if  $x = vz$ for some string $z$}
\end{array}
\right.
$$
The non-zero idempotents in $P_{n}$ are the elements of the form
$xx^{-1}$,
where $x$ is positive,
and the natural partial order is given by 
$yx^{-1} \leq vu^{-1}$ iff $(y,x) = (v,u)p$ for some positive string
$p$.
See \cite{Law1,Law2007a,Law2007b} for more about the polycyclic inverse monoids,  and proofs of any claims.

We may obtain an isomorphic copy of $P_{n}$ as an inverse submonoid of $I(A^{\omega})$ as follows.
Let $yx^{-1} \in P_{n}$.
Then define a map from $xA^{\omega}$ to $y^{A^{\omega}}$ by $xw \mapsto yw$ where $w$ is any right-infinite string.
Thus $yx^{-1}$ describes the process {\em pop the string $x$} and then {\em push the string $y$.}

\begin{remark}{\em  
In what follows, we shall always regard $P_n$ as an inverse submonoid of $I(A^{\omega})$. }
\end{remark}

We now construct a larger inverse monoid containing this copy of $P_{n}$.
The inverse monoid $I(A^{\omega})$ is a Boolean inverse monoid.
Thus finite non-empty compatible subsets have joins.
Let $S \subseteq I(A^{\omega})$ be an inverse submonoid containing zero.
Then we may form the subset $S^{\vee}$ consisting of all joins of  finite non-empty compatible subsets of $S$.
It is routine to check that $S^{\vee}$ is again an inverse submonoid of $I(A^{\omega})$.
We apply this construction to $P_{n}$ to obtain the inverse submonoid $P_{n}^{\vee}$.

\begin{lemma}\label{lem:narnia} Let $yx^{-1}$ and $vu^{-1}$ be a compatible pair of elements in the polycyclic inverse monoid $P_{n}$.
If they are not orthogonal, then either $yx^{-1} \leq vu^{-1}$ or vice-versa.
\end{lemma}
\begin{proof} Without loss of generality, suppose that $xy^{-1}vu^{-1} \neq 0$.
Then $y$ and $v$ are prefix-comparable.
Again, without loss of generality, we may assume that $y = vz$ for some $z$.
Then $xy^{-1}vu^{-1} = x(uz)^{-1}$.
But this is supposed to be an idempotent and so $x = uz$.
We substitute this into $yx^{-1}uv^{-1}$ to get $y(vz)^{-1}$.
But this too is supposed to be an idempotent and so $y = vz$.
We have therefore proved that $yx^{-1} \leq vu^{-1}$.
\end{proof}

From the above lemma, a finite non-empty compatible subset of $P_{n}$ will have the same join as  a finite non-empty orthogonal subset
obtained by taking the maximal elements of the compatible subset. 

\begin{lemma}\label{lem:orthogonal} A subset
$$\{y_{1}x_{1}^{-1}, \ldots, y_{m}x_{m}^{-1}\}$$
of $P_{n}$ is orthogonal iff $\{x_{1}, \ldots, x_{m}\}$ and
$\{y_{1}, \ldots, y_{m}\}$ are both prefix sets.
\end{lemma}

It follows that the elements of $P_{n}^{\vee}$ can be represented in the following form.
Let $x_{1} + \ldots + x_{r}$ and $y_{1} + \ldots + y_{r}$ be two prefix sets with the same number of elements.
Define a map from $(x_{1} + \ldots, + x_{r})A^{\omega}$ to $(y_{1} + \ldots + y_{r})A^{\omega}$,
denoted by,
$$
\left(
\begin{array}{ccc}
x_{1} & \ldots & x_{r}\\
y_{1} & \ldots & y_{r}
\end{array}
\right)
$$
that does the following: $x_{i}w \mapsto y_{i}w$, 
where $w$ is any right-infinite string.
We denote the totality of such maps by $C_{n} = P_{n}^{\vee}$.
We call this the {\em Cuntz inverse monoid (of degree $n$).}
We shall call the unique countable atomless Boolean algebra the {\em Tarski algebra}.
The following was proved in \cite{Law2007b}.
But we shall give the details below.
Recall that an inverse semigroup with zero is {\em $0$-simple} if there are only two ideals.
It is well-known that a $0$-simple, fundamental Boolean inverse monoid is congruence-free.

\begin{proposition} $C_{n}$ is a Boolean inverse $\wedge$-monoid whose semilattice of idempotents is the Tarski algebra.
It is fundamental, $0$-simple and has $n$ $\mathscr{D}$-classes.
It is therefore congruence-free.
Its group of units is the Thompson group $V_{n}$.
\end{proposition}
\begin{proof} We show first that we have a groupoid.
Suppose that $f \colon XA^{\omega} \rightarrow YA^{\omega}$ be such that there is a bijection
$f_{1} \colon X \rightarrow Y$ such that $f(xw) = f_{1}(x)w$ for any infinite string $w$.
Suppose that $X^{+} = x_{1}A + x_{2} + \ldots + x_{r}$.
Let $Y^{+} = y_{1}A + y_{2} + \ldots + y_{r}$.
Define $f^{+} \colon X^{+} \rightarrow Y^{+}$ as follows.
Let $f^{+}(x_{i}) = y_{i}$ for $2 \leq i \leq r$.
Define $f^{+}(x_{1}a_{j}) = y_{1}a_{j}$ for $1 \leq j \leq n$.
It is clear that $f^{+} = f$. 
We shall call it a {\em refinement} of $f$.
Let $g \colon UA^{\omega} \rightarrow VA^{\omega}$ and suppose that $XA^{\omega} = VA^{\omega}$.
Let $X'$ be obtained from $X$ by a sequence of extensions.
let $Y'$ be obtained from $Y$ by a sequence of extensions.
By Lemma~\ref{lem:uniformatizing}, we suppose that $X'$ and $Y'$ are both uniform of the same length. 
We construct $Y'$ and $U'$ by using the corresponding extensions.
It follows by Corollary~\ref{cor:aslan} that $X' = Y'$.
Let $f'_{1}$ be obtained from $f_{1}$ by successive appropriate refinments, 
and likewise let $g'_{1}$ be obtained from $g_{1}$.
Thus $f' = f$ and $g' = g$.
But we may now compose $f_{1}g_{1}$ directly to get a map from $U'$ to $Y'$ that represents $fg$.
Since inverses pose no problems, we have shown that we have a groupoid.
The semilattice of idempotents is just the Tarski algebra.
We show that this is an ordered groupoid, and so inductive, from which we get that it is an inverse monoid.
Let $f \colon XA^{\omega} \rightarrow YA^{\omega}$ where $f_{1} \colon X \rightarrow Y$ is a bijection.
Let $ZA^{\omega} \subseteq XA^{\omega}$.
Assume first that each element $z \in Z$ can be written $z = xu$ for some $x \in X$ and string $u$.
Observe that under this assumption, $x$ will be unique.
Define $g_{1}(z) = f_{1}(x)u$.
Put $Y'$ equal to the set of all $f_{1}(x)u$ as $z \in Z$.
Then $Y'A^{\omega} \subseteq YA^{\omega}$ and we have defined a bijection
$g \colon ZA^{\omega} \rightarrow Y'A^{\omega}$ which is the restriction of $f$.
It remains to show that we can verify our assumption.
This can be achieved as in Lemma~\ref{lem:uniformatizing} by using a sequence of extensions to convert $Z$ into a prefix set where
all strings have lengths strictly larger than the longest string in $X$.
Then by Lemma~\ref{lem:piglet}, since $ZA^{\omega} \subseteq XA^{\omega}$,
we have that each element of $Z$ is prefix-comparable with an element of $X$.
From length considerations, it follows that each $z \in Z$ has as a prefix an element of $X$.

It is straightforward to see that $C_{n}$ is a Boolean inverse monoid and that it is also a $\wedge$-monoid.

We now prove that $C_{n}$ is $0$-simple.
Let $X$ and $Y$ be any two prefix sets.
let $y \in Y$.
Then $yX$ is a prefix set with the same cardinality as $X$.
It follows that there is an element $f \colon XA^{\omega} \rightarrow yXA^{\omega}$ of $C_{n}$.
But $yXA^{\omega} \subseteq YA^{\omega}$.
This proves the claim.

We now prove that there are $n-1$ non-zero $\mathscr{D}$-classes.
The first step is to calculate the number of strings in a {\em maximal} prefix set.
For a fixed $n \geq 2$, and for $r = 0,1,2, \ldots$, we can construct maximal prefix sets containing $P^{n}_{r} = (r-1)n - (r-2)$ strings.
Concrete examples of such sets can be constructed by starting with the `seeds' $\varepsilon$ and $A$
and then growing maximal prefix sets by attaching $A$ from left-to-right.
We designate these specific maximal prefix sets by $M^{n}_{r}$.
There are $n-2$ numbers between $P^{n}_{r}$ and $P^{n}_{r+1}$.
Consider now the $n-2$ prefix sets
$C_{1}^{n} = a_{1} + a_{2}$,
$C_{2}^{n} = a_{1} + a_{2} + a_{3}$,
\ldots
$C_{n-2}^{n} = a_{1} + \ldots + a_{n-1}$.
The partial identities associated with $C_{i}$ and $C_{j}$ where $i \neq j$ are not $\mathscr{D}$-related.
There are therefore at least $n$ $\mathscr{D}$-classes when we add in the zero and the identity.
We may attach a copy of $M^{n}_{r}$ to the rightmost vertex of $C_{i}$.
We denote this prefix set by $C_{i} \ast M^{n}_{r}$.
Observe that $C_{i}A^{\omega} = C_{i} \ast M^{n}_{r}$.
Let $X$ be an arbitrary prefix set.
Either it is in bijective correspondence with one of the $M^{n}_{r}$, in which case the identity function on $XA^{\omega}$ is $\mathscr{D}$-related to the identity,
or it is in bijective correspondence with one of the $C_{i} \ast M^{n}_{r}$, in which case the identity function on $XA^{\omega}$ is $\mathscr{D}$-related to the identity function on $C_{i}$.
In particular, we see that $C_{2}$ is bisimple.

The group of units of $C_{n}$ consists of those elements
$$
\left(
\begin{array}{ccc}
x_{1} & \ldots & x_{r}\\
y_{1} & \ldots & y_{r}
\end{array}
\right)
$$
where $x_{1} + \ldots + x_{r}$ and $y_{1} + \ldots + y_{r}$ are maximal prefix codes.
These are precisely the elements of Thompson's group $V_{n}$.
\end{proof}

%%%%%%%%%%%%%%%%%%%%%%%%%%%%%%%%%%%%%%%%%%%%%%%%%%%%%%%%%%%%%%%%%%%%%%%%%%%%%%%%%%%
\begin{center}
{\bf The dyadic (or CAR) inverse monoid}
\end{center}

We shall need to work with measures on the Cantor set.
The general theory of such measures is the subject of current research, see \cite{Akin1999,Akin2004,ADMY,BH}, for example
but the measures we need are well-known.

Let $A$ be an alphabet with $n$ elements.
Define $\mu (a) = \frac{1}{n}$ for any $a \in A$ and define $\mu (\varepsilon) = 1$.
If $x \in A^{\ast}$ is any string of length $r$ define $\mu (x) = \frac{1}{n^{r}}$.
If $X$ is any prefix set, define $\mu (X) = \sum_{x \in X} \mu (x)$.
The following is proved as \cite[Theorem~I.4.2]{PP}.

\begin{lemma}\label{lem:acre} 
For any prefix set $X$, we have that $\mu (X) \leq 1$.
\end{lemma}

Let $U$ be any clopen subset of $A^{\omega}$.
Suppose that $U = XA^{\omega}$.
Define $\mu (U) = \mu (X)$.
We call $\mu$ defined in this way on the clopen subsets of $A^{\omega}$ the {\em Bernoulli measure}.
This measure is sometimes denoted $\beta(\frac{1}{n})$.

\begin{lemma} \mbox{}
\begin{enumerate}

\item Let $X$ be a prefix set. Then $\mu (X) = 1$ if, and only if $X$ is a maximal prefix set.

\item  The Bernoulli measure is well-defined.

\end{enumerate}
\end{lemma}
\begin{proof} (1) let $X$ be a maximal prefix set.
It is obtained by means of a sequence of extensions from $\varepsilon$ and $\mu (\varepsilon) = 1$.
Clearly, $\mu (A^{r}) = 1$.
Thus if $Y_{1}$ and $Y_{2}$ are prefix sets and $Y_{2}$ is an extension of $Y_{1}$ then
$\mu (Y_{2}) = \mu (Y_{1})$.
The result follows.
Suppose now that $\mu (X) = 1$.
If $X$ is not maximal, then we can find a string $u$ such that $X + u$ is a prefix set.
But $\mu (X + u) = \mu (X) + \mu (u) > 1$, which is a contradiction.

(2) This follows by Proposition~\ref{prop:wol}.
\end{proof}

The following result will be important later.

\begin{lemma}\label{lem:property_B} Let $A$ be an alphabet with $n \geq 2$ elements.
Let $U = XA^{\omega}$ and $V = YA^{\omega}$ be such that $X$ has length $l$, 
and $Y$ has length $m$.
Without loss of generality, we may assume that $m \geq l$.
Suppose that $\mu (U) = \mu (V)$.
Then 
there is a prefix set $X'$ uniform of length $m$ such that $U = X'A^{\omega}$,
and   
there is a prefix set $Y'$ uniform of length $m$ such that $V = Y'A^{\omega}$,
such that $\left| X' \right| = \left| Y' \right|$.
\end{lemma}
\begin{proof} By Lemma~\ref{lem:uniformatizing}, we may find a prefix set $X'$,
uniform of length $m$, such that $U = X'A^{\omega}$.
Observe that $\mu (X) = \mu (X')$.
Let $r$ be the number of strings in $X'$.
Then $\mu (X) = \frac{r}{n^{l}}$. 
Similarly, 
we may find a prefix set $Y'$,
uniform of length $m$, such that $V = Y'A^{\omega}$.
Observe that $\mu (Y) = \mu (Y')$.
Let $s$ be the number of strings in $Y'$.
Then $\mu (Y) = \frac{s}{n^{l}}$.
It follows immediately that $r = s$, as required.
\end{proof}

%%%%%%%%%%%%%%%%%%%%%%%%%%%%%%%%%%%%%%%%%%%%%%%%%%%%%%%%%%%%%%%%%%%%%%%%%%%%%%%%%%%%%%%%%%%%%%%%%%%%%%%%%%%%%%%%%%%%%%%%%
The following result was first proved in  \cite{Law2009} but suggested by earlier work of Meakin and Sapir \cite{MS}.
It shows how to construct inverse submonoids of the polycyclic inverse monoid.
A {\em wide} inverse subsemigroup of $S$ is one that contains all the idempotents of $S$.

\begin{proposition}\label{prop:bingo} Let $A$ be an $n$-letter alphabet.
Then there is a bijection between right congruences on $A^{\ast}$ and wide inverse submonoids of $P_{n}$.
If $\rho$ is the right congruence in question, then the corresponding inverse submonoid of $P_{n}$ simply
consists of $0$ and all elements $yx^{-1}$ where $(y,x) \in \rho$.
\end{proposition}

Consider now the congruence $\lambda$ of the length map $A^{\ast} \rightarrow \mathbb{N}$ given by $x \mapsto \left| x \right|$.
Define $G_{n} \subseteq P_{n}$ to consist of zero and all elements $yx^{-1}$ where $\left| y \right| = \left| x \right|$.
Then by Proposition~\ref{prop:bingo},  $G_n$ is an inverse monoid.
It was first defined in the thesis of David Jones \cite{Jones} and is called the {\em gauge inverse monoid (on $n$ letters)}
and arose from investigations of strong representations of the polycyclic inverse monoids \cite{JL1} motivated by the theory developed in \cite{BJ}.

We now define $Ad_{n} \subseteq C_{n}$, called the {\em $n$-adic} inverse monoid.
In the case $n = 2$, we refer to the {\em dyadic inverse monoid}.
By definition, its consists of those elements of $C_{n}$ which are orthogonal joins of elements of $G_{n}$.
That is, maps of the form
$$
\left(
\begin{array}{ccc}
x_{1} & \ldots & x_{r}\\
y_{1} & \ldots & y_{r}
\end{array}
\right)
$$
where $\left| y_{i} \right| = \left| x_{i} \right|$ for $1 \leq i \leq r$.
The proof of the following is immediate.

\begin{proposition} 
The $n$-adic inverse monoid is a fundamental Boolean inverse monoid and wide inverse submonoid of the Cuntz inverse monoid $C_{n}$.
\end{proposition}

The following result will establish most of the properties we shall need to prove our main theorem.

\begin{proposition} 
The dyadic inverse monoid may be equipped with a good invariant mean that reflects the $\mathscr{D}$-relation.
\end{proposition}
\begin{proof} 
The idempotents of $A_{2}$ are simply the clopen subsets of the Cantor space.
We equip these with the Bernoulli measure $\beta(\frac{1}{2})$.
We show first that $\mu$ is an invariant mean.
There is only one property we have to check.
Let $\mathbf{e}$ and $\mathbf{f}$ be $\mathscr{D}$-related idempotents in $A_{2}$.
Let $\mathbf{e}$ be the identity function on the clopen subset $U$ and
let $\mathbf{f}$ be the identity function on the clopen subset $V$.
Then there are prefix sets $X = x_{1} + \ldots + x_{r}$ and $Y = y_{1} + \ldots + y_{r}$ such that
$U= (x_{1} + \ldots + x_{r} )A^{\omega}$ and $V = (y_{1} + \ldots + y_{r})A^{\omega}$ such that
$y_{i}x_{i}^{-1}$ are elements of the gauge inverse monoid.
That is, we have a map
$$
\left(
\begin{array}{ccc}
x_{1} & \ldots & x_{r}\\
y_{1} & \ldots & y_{r}
\end{array}
\right)
$$
where $\left| y_{i} \right| = \left| x_{i} \right|$ for $1 \leq i \leq r$
from $\mathbf{e}$ to $\mathbf{f}$.
In particular, the sets $X$ and $Y$ contain the same number of strings,
and the same number of strings of the same length.
It is now immediate that $\mu (\mathbf{e}) = \mu (\mathbf{f})$.
The fact that the $\mathscr{D}$-relation is reflected follows from Lemma~\ref{lem:property_B}.

It remains to prove that this invariant mean is good.
Let $\mu (\mathbf{e}) \leq \mu(\mathbf{f})$.
We work with clopen subsets and so we assume that $\mu (U) \leq \mu (V)$.
This may be easily deduced using Lemma~\ref{lem:uniformatizing} and a modified version of  Lemma~\ref{lem:property_B}.
\end{proof}

The above proposition, combined with Lemma~\ref{lem:5April}, tells us that the dyadic inverse monoid is a Foulis monoid
and that its lattice of principal ideals forms a linearly ordered set isomorphic to the dyadic rationals in the unit interval.
We therefore now have the main result of this section.

\begin{theorem} 
The MV-algebra of dyadic rationals is co-ordinatized by the dyadic inverse monoid.
\end{theorem}

%%%%%%%%%%%%%%%%%%%%%%%%%%%%%%%%%%%%%%%%%%%%%%%%%%%%%%%%%%%%%%%%%%%%%%%%%%%%%%%%%%%%%%%%%%%%%%%%%%%%%%%%%%%%%%%%%%%%%%%%
It is worth looking in more detail at the structure of the dyadic inverse monoid $Ad_{2}$.

\begin{proposition} The dyadic inverse monoid is isomorphic to the direct limit of the sequence
$$I_{1} \rightarrow I_{2} \rightarrow I_{4} \rightarrow I_{8} \rightarrow \ldots$$
It is therefore an AF inverse monoid.
\end{proposition}
\begin{proof} Let $A = a + b$.
We construct the binary tree with root $A^{\omega}$ and then vertices
$aA^{\omega}$ and $bA^{\omega}$ at the first level,
$aaA^{\omega}$, $abA^{\omega}$, $baA^{\omega}$ and $bbA^{\omega}$ at the second level, and so on.
The clopen sets at each level are pairwise disjoint.
Every clopen set has a generating set constructed from taking the union of the above sets at the same level.
This is a result of Lemma~\ref{lem:uniformatizing}.
However, the same subset can, of course, be represented in different ways.
Thus the clopen set $aA^{\omega}$ which is from level 1, can also be written as $aaA^{\omega} + abA^{\omega}$, a union of sets constructed from level 2.

We now observe that the elements of $Ad_{2}$ constructed from the gauge inverse monoid maps at level $l$ form an inverse monoid isomorphic to $I_{2^{l}}$.
The best way to see this is that at level $l$ we may construct all the relevant matrix units together with the identity and the zero.
For example, at level 2, we have, in addition to the identity and the zero, the 4 idempotents
$$aa(aa)^{-1}, ab(ab)^{-1}, ba(ba)^{-1}, bb(bb)^{-1}$$
and then the non-identity matrix units such as $aa(ab)^{-1}$.
By taking joins we get all the other elements of $I_{2^{l}}$.
In addition, we see that this copy is actually an inverse submonoid of $Ad_{2}$ containing the zero.

We claim next that $I_{2^l} \subseteq I_{2^{l+1}}$.
This is also best seen by focusing on the matrix units.
First observe that every idempotent of level $l$ is also an idempotent at level $l+1$.
If $XA^{\omega}$ is a clopen subset with $X$ a union of idempotents at level $l$
then $XA^{\omega} = XaA^{\omega} + XbA^{\omega}$.
It follows that every element of $I_{l}$ reappears in $I_{l+1}$ by the process of refinement.

It is now evident that $Ad_{2} = \bigcup_{l=1}^{\infty} I_{2^{l}}$, which proves the theorem.
\end{proof}

\begin{remark} 
{\em In the light of the above result, we might also call the dyadic inverse monoid the {\em CAR inverse monoid}.}
\end{remark}

The group of units of $Ad_{2}$ is the direct limit $S_{1} \rightarrow S_{2} \rightarrow S_{4} \rightarrow \ldots$
where the inclusions between successive symmetric groups are block diagonal maps.

\section{Proof of the main theorem}

The goal of this section is to prove the following.

\begin{theorem}[Co-ordinatization]\label{the:one} 
Let $E$ be a countable MV-algebra.
Then there is a Foulis monoid $S$ satisfying the lattice condition such that $S/\mathscr{J}$  is isomorphic to $E$.
\end{theorem}

We begin by giving some standard definitions and results we shall need.

%%%%%%%%%%%%%%%%%%%%%%%%%%%%%%%%%%%%%%%%%%%%%%%%%%%%%%%%%%%%%%%%%%%%%%%%%%%%%%%%%%%%%%%%%%%%%%%%%%%%%%%%%%%
An {\em ordered abelian group} $G$ is given by a submonoid $G^{+} \subseteq G$ called the {\em positive cone}
such that $G^{+}  \cap (-G^{+}) = \{0\}$ and $G = G^{+} - G^{+}$.
If $a,b \in G$ define $a \leq b$ if, and only if, $b - a \in G^{+}$.
The condition $G = G^{+} - G^{+}$ means that $G$ is the group of fractions of its positive cone.
The condition $G^{+}  \cap (-G^{+}) = \{0\}$ means that $0$ is the only invertible element of $G^{+}$.
We say that $G^{+}$ is {\em conical} if it has trivial units.
The theory of abelian monoids tells us that every abelian conical cancellative monoid arises as the
positive cone of an ordered abelian group.
If the order in a partially ordered abelian group $G$ actually induces a lattice structure on $G$ we say that the group is {\em lattice-ordered} or an {\em $l$-group}.

Let $G$ be a partially ordered abelian group.
An {\em order unit} is a positive element $u$ such that for any $g \in G$ there exists a natural number $n$ such that $g \leq nu$.
Let $u \in G$ be any positive, non-zero element.
Denote by $[0,u]$ the set of all elements of $g$ such that $0 \leq g \leq u$.
The notation is not intended to suggest that this set is linearly ordered.
Let $p,q \in [0,u]$.
Define the partial binary operation $\oplus$ on $[0,u]$ by $p \oplus q = p + q$ if $p + q \in [0,u]$,
and undefined otherwise.
If $p \in [0,u]$ define $p' = u - p$.
Then $[0,u]$ becomes an effect algebra \cite[Theorem~3.3]{FB1994}.
We call this the {\em interval effect algebra} associated with $(G,u)$.
If, in addition,  $G$ is an $l$-group and $u$ is an order-unit, then $[0,u]$ is actually an MV-algebra.
The following is proved in \cite[Theorem 3.9]{M1}, \cite[Corollary 7.1.8]{CDM} and \cite{M3}.

\begin{theorem}\label{the:mundici} 
Every MV-algebra is isomorphic to an interval effect algebra $[0,u]$
where $u$ is an order unit in an $l$-group.
\end{theorem}

We briefly sketch out how the above theorem may be proved.
If $(E,\oplus,0)$ is a partial algebra, then we may construct its {\em universal monoid} $\nu \colon E \rightarrow M_{E}$ in the usual way.
However,  we are interested not merely in the existence of $M_{E}$ but in its properties so we shall give more details on how the universal monoid is constructed.
The proof of part (1) below follows from \cite{Baer} and \cite[Lemma~1.7.6, Proposition~1.7.7, Proposition~1.7.8, Lemma~1.7.10, Lemma~1.7.11, Theorem~1.7.12]{DP}.
It is noteworthy that commutativity arises naturally and does not have to be imposed.
The proof of part (2) below follows from \cite[Theorem~1.7.12]{DP}.
Alternative approaches can be found in \cite{GW,Wehrung}.

\begin{proposition}\label{prop:free} Let $(E,\oplus,0)$ be a conical partial refinement monoid.
\begin{enumerate}

\item  Let $E^{+}$ denote the free semigroup on $E$.
Define $\sim$ to be the congruence on $E^{+}$ generated by $(a,b) \sim (a \oplus b)$ when $\exists a \oplus b$.
Put $M = E^{+}/\sim$.
Then $M$ is a conical abelian monoid and is the universal monoid of $E$.

\item Suppose that $(E,\oplus,0,1)$ is also an effect algebra.
Then $M$ is cancellative, 
the image of $E$ in $M$ is convex,
and the image of $1$ in $M$ is an order unit.

\end{enumerate}
\end{proposition} 

An abelian monoid always has a universal group:  its {\em Grothendieck group}.
If the abelian monoid is cancellative and conical then its Grothendieck group is partially ordered and is its group of fractions.
It follows that the Grothendieck group of the universal monoid of an effect algebra satisfying the refinement property is the {\em universal group}
of that effect algebra.
This leads to the main theorem we shall need proved by Ravindran \cite{R}.
Its full proof may be found as \cite[Theorem 1.7.17]{DP}.

\begin{theorem}[Ravindran]\label{the:ravindran} Let $E$ be an effect algebra satisfying the refinement property.
\begin{enumerate}

\item The universal group $\gamma \colon E \rightarrow G_{E}$ is a partially ordered abelian group with the refinement property.
Its positive cone $P$ is generated as a submonoid by the image of $E$ under $\gamma$.

\item Put $u = \gamma (1)$. Then $u$ is an order unit in $G_{E}$ and $E$ is isomorphic under $\gamma$ to the interval effect algebra $[0,u]$.

\item If $E$ is actually an MV-algebra, then $[0,u]$ is a lattice from which it follows that $G_{E}$ is an $l$-group.
If $E$ is countable then $G_{E}$ is countable.

\end{enumerate}
\end{theorem}

The proof of the following is immediate but it is significant from the point of view of the main goal of this paper.

\begin{proposition}\label{prop:important} Let $S$ be a Foulis monoid.
Then $S/\mathscr{J}$ is isomorphic to the interval $[0,u]$ where $u$ is an order unit in the universal group of $\mathsf{E}(S)$
and is the image of the class of the identity of $E(S)$.
\end{proposition}

Every AF inverse monoid is a Foulis monoid by Theorem~\ref{the:holiday}.
Accordingly, our first aim will be to explicitly compute the universal group of the effect algebra associated with an AF inverse monoid.
To do this, it will be useful to work with the idempotents of the inverse monoid directly rather than with the elements of the associated effect algebra.
This is the import of the following definition.

Let $S$ be a Boolean inverse monoid.
A {\em group-valued invariant mean} on $S$ is a function $\theta \colon E(S) \rightarrow G$ to an abelian group $G$ such
that the following two axioms hold:
\begin{description}
\item[{\rm (GVIM1)}] If $e$ and $f$ are orthogonal then $\theta (e \vee f) = \theta (e) + \theta (f)$.

\item[{\rm (GVIM2)}]  We have that $\theta (s^{-1}s) = \theta (ss^{-1})$ for all $s \in S$.
\end{description}
It follows from (GVIM1) that $\theta (0) = 0$.
By the usual considerations, a universal group-valued invariant mean always exists.

The following lemma tells us that we can, indeed, pull-back to the set of idempotents of the inverse monoid.

\begin{lemma}\label{le:hols} Let $S$ be a Foulis monoid.
Then the universal group-valued invariant mean is the universal group of the associated effect algebra.
\end{lemma} 
\begin{proof}
Let $\nu \colon E(S) \rightarrow G_{S}$ be the universal group-valued invariant mean.
Denote by $\nu' \colon \mathsf{E}(S) \rightarrow G_{S}$ by $\nu' ([e]) = [\nu (e)]$.
This is a well-defined map such that if $[e] \oplus [f]$ exists then $\nu' ([e] + [f]) = \nu'([e]) + \nu' ([f])$. 
Because of axiom (GVIM2), we may define a function
$\mu \colon \mathsf{E}(S) \rightarrow G_{S}$ by $\mu ([e]) = \nu (e)$.
Suppose that $[e] \oplus [f]$ is defined.
Then it equals $[e' \vee f']$ where $e \mathscr{D} e'$ and $f \mathscr{D} f'$.
But $\nu (e' \vee f') = \nu (e') + \nu (f')$, and so $\mu ([e]) + \mu ([f]) = \mu ([e] \oplus [f])$.

Now let $\theta \colon \mathsf{E}(S) \rightarrow H$ be any map to a group such that
if $[e] \oplus [f]$ is defined then $\theta ([e] \oplus [f]) = \theta ([e]) + \theta ([f])$.
Define $\phi \colon E(S) \rightarrow H$ by $\phi (e) = \theta ([e])$.
Then it is immediate that $\phi$ is a group-valued invariant mean.
It follows that there is a group homomorphism $\alpha \colon G_{S} \rightarrow H$ such that $\alpha \nu = \phi$. 
Clearly, $\alpha \nu' = \theta$.
\end{proof}

%%%%%%%%%%%%%%%%%%%%%%%%%%%%%%%%%%%%%%%%%%%%%%%%%%%%%%%%%%%%%%%%%%%%%%%%%%%%%%%%%%%%%%%%%%%%%%%%%%%%%%%%%%%%%%%%%%
We now set about computing the universal group-valued mean of an AF inverse monoid.
First, we shall need some definitions.
A {\em simplicial group} is simply a group of the form $\mathbb{Z}^{r}$ with the usual ordering.
A {\em positive homomorphism} between simplicially ordered groups maps positive elements to positive elements.
If the ordered groups are also equipped with distinguished order units, then a homomorphism is said to be {\em normalized}
if it maps distinguished order units one to the other.
A {\em dimension group} is defined to be a direct limit of a sequence of simplicially ordered groups and positive homomorphisms.
An ordered abelian group   is said to satisfy the {\em Riesz interpolation property} (RIP) if $a_{1},a_{2} \leq b_{1},b_{2}$,
in all possible ways, implies that there is an element $c$ such that $a_{1},a_{2} \leq c$ and $c \leq b_{1},b_{2}$.
Such a group satisfies the {\em Riesz decomposition property} (RDP) if for all {\em positive} $a,b,c$ if $a \leq b + c$
implies that there are positive elements $b',c'$ such that $b' \leq b$ and $c' \leq c$ and $a = b' + c'$.
These two properties (RIP and RDP) are equivalent for partially ordered abelian groups \cite[Proposition~21.3]{Goodearl}
(but not for effect algebras).
The partially ordered abelian group $(G, G^+)$ is said to be {\em unperforated} if $g \in G$ and $ng \in G^+$ for some natural number $ n \geq 1$ implies that $g \in G^+$.
The proof of part (1) of the following is part of \cite[Theorem~3.1]{Effros}, and the proof of part (2) is from \cite[Corollary 21.9]{Goodearl}

\begin{theorem}\label{the:dim} \mbox{}
\begin{enumerate}

\item Countable partially ordered abelian groups are dimension groups precisely when they satisfy the Riesz interpolation property and are unperforated.

\item Each countable dimension group with a distinguished order unit is isomorphic to a direct limit of a sequence of simplicial groups with order-units and normalized positive homomorphisms.
Thus each such group is constructed from a Bratteli diagram.

\item Countable $l$-groups are dimension groups.

\end{enumerate}
\end{theorem}  

%%%%%%%%%%%%%%%%%%%%%%%%%%%%%%%%%%%%%%%%%%%%%%%%%%%%%%%%%%%%%%%%%%%%%%%%%%%%%%%%%%%%%%%%%%%%%%%%%%%%%%%%%%%%%%
We may now explicitly compute the universal group-valued invariant means of AF inverse monoids.
We begin with a special case.
In what follows, we denote by $\left| e \right|$ the cardinality of the set $A$ where $e = 1_{A}$.

\begin{lemma}\label{lem:ww1} \mbox{} 
\begin{enumerate}
\item Let $I_{n}$ be a finite symmetric inverse monoid on $n$ letters.
Define the function $\pi \colon E(I_{n}) \rightarrow \mathbb{Z}$ by $\pi (1_{A}) = \left| A \right|$.
Then $\pi$ is the universal group-valued invariant mean of $I_{n}$ and the image of the identity is $n$, an order unit.

\item Let $T = S_{1} \times \ldots \times S_{r}$  be a semisimple inverse monoid, where $n(1), \ldots, n(r)$ are the number of letters in the underlying sets of $S_{1}, \ldots, S_{r}$, respectively. 
Put $\mathbf{n} = (n(1), \ldots, n(r))$.
Define 
$$\pi \colon E(S_{1} \times \ldots \times S_{r}) \rightarrow \mathbb{Z}^{r}$$ 
by 
$$\pi (e_{1}, \ldots, e_{r}) = (\left| e_{1} \right|, \ldots, \left| e_{r} \right|).$$
Then $\pi$ is the universal group-valued invariant mean of $T$
and the identity of $T$ is mapped to the order unit $\mathbf{n}$.
\end{enumerate}
\end{lemma}
\begin{proof} (1) It is straightforward to check that $\pi$ has the requisite properties.
The universal property follows from the fact that the atoms of $E(I_{n})$ are mapped to the identity of $\mathbb{Z}$.
The proof of (2) follows from (1).
\end{proof}

We may now prove the general case.

\begin{proposition}\label{prop:key_link} Let $B$ be a Bratteli diagram with associated AF inverse monoid $\mathsf{I}(B)$ and associated dimension group $\mathsf{G}(B)$.
Then the universal group-valued invariant mean of $\mathsf{I}(B)$ is given by a map $\pi \colon E(\mathsf{I}(B)) \rightarrow \mathsf{G}(B)$
where the image of the identity of $E(\mathsf{I}(B))$ is an order unit $u$ in $\mathsf{G}(B)$. 
\end{proposition}
\begin{proof} From the Bratteli diagram $B$, we may construct a sequence
$$T_{0} \stackrel{\sigma_{0}}{\rightarrow} T_{1} \stackrel{\sigma_{1}}{\rightarrow} T_{2} \stackrel{\sigma_{2}}{\rightarrow} \ldots $$
of semisimple inverse monoids and injective standard morphisms.
By definition,  $\mathsf{I}(B) = \varinjlim T_{i}$.
Observe that $E(\mathsf{I}(B)) =  \varinjlim E(T_{i})$.
We begin by defining a map $\pi \colon E(\mathsf{I}(B)) \rightarrow \mathsf{G}(B)$,
that will turn out to have the required properties.
We consider level $i$ of the Bratteli diagram $B$.
The semisimple inverse monoid $T_{i}$ is a product $S_{1} \times \ldots \times S_{r(i)}$ of $r(i)$ symmetric inverse monoids,
where $n(1), \ldots, n(i)$ is the number of letters in the underlying sets of $S_{1}, \ldots, S_{r(i)}$, respectively. 
Put $\mathbf{n}(i) = (n(1), \ldots, n(i))$.
Define 
$$\pi_{i} \colon E(S_{1} \times \ldots \times S_{r(i)}) \rightarrow \mathbb{Z}^{r(i)}$$ 
as in Lemma~\ref{lem:ww1}.
Then also by Lemma~\ref{lem:ww1},  $\pi_{i} \colon E(T_{i}) \rightarrow \mathbb{Z}^{r(i)}$ is the universal group-valued invariant mean of $T_{i}$
and the identity of $T_{i}$ is mapped to the order unit $\mathbf{n}(i)$.
Let $\beta_{i} \colon \mathbb{Z}^{r(i)} \rightarrow \mathbb{Z}^{r(i+1)}$ be the $r(i+1) \times r(i)$ matrix
defined after Remark~\ref{rem:noodle}.
We also denote by  $\sigma_{i}$ the restriction of that map to $E(T_{i})$.
We claim that $\beta_{i} \pi_{i} = \pi_{i+1} \sigma_{i}$ and that it is a normalized positive homomorphism.
This follows from two special cases.
First, we consider the standard map from $R_{m}$ to $R_{n}$  given by $A \mapsto sA$.
If $A$ represents an idempotent then $\left| A \right|$ is simply the number of 1's along the diagonal.
Clearly, $\left| sA \right| = s \left| A \right|$.
Thus the corresponding map $\beta$ from $\mathbb{Z}$ to $\mathbb{Z}$ is simply multiplication by $s$.
Observe that $sm = n$.
Second, we consider the standard map from $R_{m(1)} \times \ldots \times R_{m(k)}$ to $R_{n}$
given by $(A_{1},\ldots, A_{k}) \mapsto s_{i1}A_{1} \oplus \ldots \oplus s_{ik}A_{k}$ where $n = s_{1}m(1) + \ldots + s_{k}m(k)$.
The corresponding map from $\mathbb{Z}^{k} \rightarrow \mathbb{Z}$ is given by the $1 \times k$-matrix
$$\left( 
\begin{array}{c c  c}
 s_{1} & \ldots & s_{k}
\end{array}
\right)$$
Our claim now follows.
Thus from the properties of direct limits that we have a well-defined map $\pi \colon E(\mathsf{I}(B)) \rightarrow \mathsf{G}(B)$, 
by construction it is a group-valued invariant mean,
and the image of the identity is an order-unit.
The fact that it has the requisite universal properties follows from the fact that each map $\pi_{i}$ has the requisite universal properties.
\end{proof}

The following theorem combines Proposition~\ref{prop:important}, Proposition~\ref{prop:key_link} and Theorem~\ref{the:ravindran} in the form that we shall need.

\begin{theorem}\label{the:nearly_there} Let $S$ be an AF inverse monoid satisfying the lattice condition.
Then the universal group-valued invariant mean $\mu \colon E(S) \rightarrow G_{S}$
is such that $G_{S}$  is a countable $l$-group and the image of the identity of $S$ in $G_{S}$ is an order unit $u$.
In addition, $S/\mathscr{J}$ is isomorphic to $[0,u]$ as an MV-algebra.
\end{theorem}

{\em We may now prove Theorem~\ref{the:one}.}
Let $E$ be a countable MV-algebra.
Then by Theorem~\ref{the:mundici} and Theorem~\ref{the:ravindran},
$E$ is isomorphic to the MV-algebra $[0,u]$ where $u$ is an order-unit in the universal group $G$ of $E$.
The group $G$ is a countable $l$-group and  by Theorem~\ref{the:dim},  it is a countable dimension group.
Thus there is a Bratteli diagram $B$ such that $\mathsf{G}(B) = G$.
Let $\mathsf{I}(B)$ be the AF inverse monoid constructed from $B$.
Then by Proposition~\ref{prop:key_link} and Theorem~\ref{the:nearly_there},
we have that $\mathsf{I}(B)/\mathscr{J}$ is isomorphic to $[0,u]$ as an MV-algebra.
Observe that $\mathsf{I}(B)/\mathscr{J}$ satisfies the lattice condition, because $[0,u]$ is a lattice.
It follows that we have co-ordinatized the MV-algebra $E$ by means of the AF inverse monoid that satisfies the lattice condition.

%%%%%%%%%%%%%%%%%%%%%%%%%%%%%%%%%%%%%%%%%%%%%%%%%%%%%%%%%%%%%%%%%%%%%%%%%%%%%%%%%%%%%%%%%%%%%
\section{Concluding remarks}

In this paper, we have shown how to co-ordinatize all countable MV-algebras,
and concretely illustrated the result with the construction of the dyadic inverse monoid.
We leave to furture work the problem of constructing concrete examples of inverse monoids 
that co-ordinatize well-known countable MV-algebras such as the rationals and algebraic numbers in $[0,1]$,
as well as the free MV-algebras on finitely many generators.
For a long list of examples of countable MV-algebras, see Table 1 of Mundici \cite{M2}.
Our theory is adapted to working with the countable case only.
This leaves completely open the question of uncountable cardinalities as well as the still more general question 
of co-ordinatizing effect algebras.

The theory of effect algebras once seemed like a niche area of research in mathematics,
but recent work has suggested that it may occupy a more central position.
In particular, the work of Jacobs \cite{Jacobs} illustrates the breadth and scope of effect algebras,
while suggesting a framework for understanding those categories which admit some kind of dimension theory.
For a general lattice-theoretic treatment of dimension theory, generalizing the work of von Neumann,
see Wehrung \cite{Wehrung}, and for some preliminary connections of effect algebras to traditional dimension groups, see \cite{JP}.
It is too early to say how our work and that of Jacobs are related,
though we might speculate that it occupies a position midway between his categories and the effect algebras.
Specifically, our work should be generalizable to inverse categories and this might lead to some insight into the connections.

Finally,  it is noteworthy that Elliott's original construction of what he calls the local semigroup associated with a $C^{\ast}$-algebra \cite{Elliott},
which the main construction of our paper parallels, is , in fact, the construction of an effect algebra.
This raises the question of whether effect algebras have the potential to provide a finer class of invariants for $C^{\ast}$-algebras.

%%%%%%%%%%%%%%%%%%%%%%%%%%%%%%%%%%%%%%%%%%%%%%%%%%%%%%%%%%%%%%%%%%%%%%%%%%%%%%%%%%%%%%%
\section{Appendix: AF inverse monoids and AF $C^{\ast}$ algebras}

This section is not needed to prove our main results.
Instead, it is intended to show that there is a closer connection between AF inverse monoids and AF $C^{\ast}$-algebras than 
merely one of analogy in the following sense:
the \'etale groupoid associated with an AF inverse monoid under non-commutative Stone duality
is the same as the groupoid associated with AF $C^{\ast}$-algebras.

%%%%%%%%%%%%%%%%%%%%%%%%%%%%%%%%%%%%%%%%%%%%%%%%%%%%%%%%%%%%%%%%%%%%%%%%%%%%%%
\subsection{Preliminaries}

If $P$ is a poset and $a \in P$, we write $a^{\downarrow}$ for the set $\{b \in P \colon b \leq a\}$ and $a^{\uparrow} =\{b \in P \colon a \leq b \}$.
A subset $Q$ of $P$ is called an {\em order ideal} if $q \in Q$ and $p \leq q$ implies that $p \in Q$.
Let $\mathsf{D}(S)$ be the inverse semigroup of all finitely generated compatible order ideals of $S$.
This is the {\em (finitary) Schein completion} of $S$.

\begin{proposition}[Schein completion]\label{prop:sc}
 The Schein completion $\mathsf{D}(S)$ of an inverse semigroup $S$ 
is a distributive inverse semigroup and the map  $\sigma \colon S \rightarrow \mathsf{D}(S)$, given by $s \mapsto s^{\downarrow}$,
is universal for homomorphisms to distributive inverse semigroups.
\end{proposition}

Let $P$ be a poset with zero.
We say that $P$ is {\em unambiguous} if whenever $a,b,c \in P$ where $a \neq 0$ such that $a \leq b,c$ then $b \leq c$ or $c \leq b$.
We say that $P$ is {\em Dedekind finite} if for each non-zero element $a \in P$ the set $a^{\uparrow}$ is finite.
An inverse semigroup with zero is said to be {\em $E^{\ast}$-unitary} if $0 \neq e \leq a$ where $e$ is an idempotent implies that $a$ is an idempotent.
The following is \cite[Lemma~2.17]{JL}.

\begin{lemma}\label{le:plod} Let $S$ be an inverse monoid with zero which is $E^{\ast}$-unitary and whose semilattice of idempotents is unambiguous.
Then the natural partial order on $S$ is unambiguous.
\end{lemma}

The following result can easily be proved directly, though a proof may be found in \cite{JL}.

\begin{lemma}\label{le:unambig} Let $S$ be an inverse monoid with zero that has an unambiguous natural partial order.
Then each finitely generated compatible order ideal of $S$ can be generated by a finite set of pairwise orthogonal elements.
\end{lemma}

\begin{remark} {\em We shall later construct the Schein completion $\mathsf{D}(S)$ of an inverse monoid
with an unambiguous order. 
By Lemma~\ref{le:unambig}, we need only consider finitely generated compatible order ideals generated by orthogonal elements.}
\end{remark}

A congruence on a semigroup with zero is said to be {\em 0-restricted} if the zero forms a congruence class on its own.
A congruence is said to be {\em idempotent-pure} if the congruence class of an idempotent only contains idempotents.

%%%%%%%%%%%%%%%%%%%%%%%%%%%%%%%%%%%%%%%%%%%%%%%%%%%%%%%%%%%%%%%%%%%%%%%%%%%%%%%%%%%%%%%%%%%%%%%%%%%
\subsection{Bratteli inverse monoids}

We introduce a second inverse monoid constructed from a Bratteli diagram $B$ that will ultimately shed light on the structure of $\mathsf{I}(B)$.

A Bratteli diagram is simply a type of (rooted) directed graph, and from any (rooted) directed graph we may construct an inverse monoid
in a way that seems first to have been employed in \cite{AH}, but has been rediscovered many times.
Let $G$ be any directed graph.
We denote by $G^{\ast}$ the free category generated by $G$.
The {\em graph inverse semigroup} $P_{G}$ consists of a zero and all symbols
$xy^{-1}$ where $x$ and $y$ are elements of $G^{\ast}$ that begin at the same vertex
together with the following multiplication
\begin{equation*}
xy^{-1} \cdot uv^{-1} =
\begin{cases}
 xzv^{-1} & \mbox{if $u=yz$ for some path $z$}\\
 x \left( vz \right)^{-1} & \mbox{if $y=uz$ for some path $z$}\\
 0 & \mbox{otherwise.}\\
\end{cases}
\end{equation*}
It can be shown that we do indeed get an inverse semigroup in this way.
The non-zero idempotents are the elements of the form $xx^{-1}$.
The natural partial order is given by
$$xy^{-1} \leq uv^{-1}
\Leftrightarrow 
\exists p \in \mathcal{G}^{\ast} \text{ such that }  
x = up \text{ and } y = vp.$$
An abstract characterization of graph inverse semigroups was given in \cite{JL}.
A directed graph $G$ is said to be {\em rooted} if there is a vertex $v_{0}$, called the {\em root}, such that given any vertex $v$ in $G$ there is a path from $v$ to $v_{0}$.
Let $G$ be a rooted directed graph with root $v_{0}$.
Define $P_{G}^{\bullet}$ to be the subset of $P_{G}$ consisting of zero and all elements $xy^{-1}$ where $x$ and $y$ both end at the root $v_{0}$.
Then, in fact,  $P_{G}^{\bullet}$ is a local submonoid of $P_{G}$ and, though we shall not need this fact here, $P_{G}$ is what is called an enlargement of $P_{G}^{\bullet}$ \cite{Jones}.
We shall denote the identity of $P_{G}^{\bullet}$ by 1.
It is equal to $1_{v_{0}}1_{v_{0}}^{-1}$.
Bratteli diagrams $B$ are rooted directed graphs.
Observe that to be concordant with our definitions above,
you should think of the edges as being directed in the reverse direction
in order that $v_{0}$ be a root, but this has little significance.
We may therefore construct the graph inverse monoid $P_{B}^{\bullet}$.
We call this the {\em Bratteli inverse monoid} constructed from the Bratteli diagram $B$.
It is an obvious question to determine the relationship between  $P_{B}^{\bullet}$ and $\mathsf{I}(B)$
and this will be our main goal.
We begin by determining some of the properties of  $P_{B}^{\bullet}$. 
To do this, we shall need the following notion.
Let $S$ be an inverse semigroup.
A function $\beta \colon S\setminus\{0\} \rightarrow \mathbb{N}$ is called a {\em weight function} if it satisfies the following axioms:
\begin{description}

\item[{\rm (W1)}] $s < t$ implies that $\beta (s) > \beta (t)$.

\item[{\rm (W2)}] $s \, \mathcal{D} \, t$ implies that $\beta (s) = \beta (t)$.

\end{description}

\begin{lemma} Let $S$ be an inverse monoid equipped with a weight function $\beta$.
\begin{enumerate}

\item $S$ is completely semisimple.

\item If $E(S)$ is unambiguous and $st \neq 0$ then $\beta (st) = \mbox{max}\{ \beta (s), \beta (t)\}$.

\end{enumerate}
\end{lemma}
\begin{proof} (1) Suppose that $e \, \mathcal{D} \, f \leq e$ for idempotents $e,f$.
Then $\beta (e) = \beta (f)$.
We are given that $f \leq e$.
But the inequality cannot be strict and so $e = f$.

(2) Suppose that $st \neq 0$. 
Put $e = s^{-1}stt^{-1} \neq 0$.
Then $st = (se)(te)$ where $st$ and $s$ and $t$ are all $\mathscr{D}$-related.
It follows that $\mu (st) = \mu (se)$.
But $(se)^{-1}se = e$ and so $\mu (st) = \mu (e)$.
Now $e = s^{-1}s \wedge tt^{-1}$.
Thus by unambiguity, we have that $s^{-1}s \leq t^{-1}t$ or $t^{-1}t \leq s^{-1}s$.
Suppose, without loss of generality, $s^{-1}s \leq t^{-1}t$.
Then $\mu (st) = \mu (s^{-1}s) = \mu (s)$.
But $\mu (s^{-1}s) \geq \mu (t^{-1}t)$ and so $\mu (s) \geq \mu (t)$.
It follows that in this case $\mu (st) = \mbox{max}\{ \mu (s), \mu (t)\}$.
\end{proof}

Part (4) below shows that Bratteli diagrams can be regarded as the posets of principal ideals of an inverse monoid.
Note that an inverse semigroup $S$ is {\em combinatorial} if $a,b \in S$ are such that $\mathbf{d}(a) = \mathbf{d}(b)$,
and $\mathbf{r}(a) = \mathbf{r}(b)$, then $a = b$.

\begin{proposition}\label{prop:boozy} Let $B$ be a Bratteli diagram.
\begin{enumerate}

\item  The inverse monoid $P_{G}^{\bullet}$ is equipped with a weight function such that $\mu^{-1}(0) = 1$ and for each $n \in \mathbb{N}$ the set $\mu^{-1}(n)$ is finite and non-empty.
If $s < t$ then there exists $s \leq t' < t$ such that $\mu (t') = \mu (t) + 1$. 

\item The inverse monoid $P_{G}^{\bullet}$  is completely semisimple, combinatorial and $E^{\ast}$-unitary.

\item The semilattice of idempotents is unambiguous, above each non-zero idempotent are only a finite number of idempotents, and there are no atoms.

\item  $B$, with a zero adjoined at the bottom, is the Hasse diagram of $P_{B}^{\bullet}/\mathscr{J}$.

\end{enumerate}
\end{proposition}
\begin{proof} Most of these results are straightforward to prove.
We simply highlight the key points.

 If $xy^{-1}$ is a nonzero element then $x$ and $y$ are paths that must begin at the same vertex of $B$ and end at the root.
It follows that $x$ and $y$ must also be the same length.
We define $\mu (xy^{-1}) = \left| x \right| = \left| y \right|$.
It is clear that (W1) holds.
The fact that (W2) holds follows from the fact that $xy^{-1} \, \mathscr{D} \, uv^{-1}$
if, and only if, $y$ and $v$ start at the same vertex.

Let $xx^{-1}$ and $yy^{-1}$ be non-zero idempotents.
Then $xx^{-1}\, \mathscr{D} \, yy^{-1}$ if and only if $x$ and $y$ begin at the same vertex $v$.
It follows that there is a bijection between the non-zero $\mathscr{D}$-classes and the vertices of the Bratteli diagram.
Let $v$ be a vertex at level $n$.
Let $P_{v}$ be the set of all strings that start at $v$ and end at the root,
remembering our convention about edge directions, here.
The number of elements in $P_{v}$ is just the size of $v$, defined earlier.
The set of all elements $xy^{-1}$ of $P_{B}^{\bullet}$ where $x,y \in P_{v}$
forms a connected principal groupoid with $\left| P_{v} \right|$ identities.
These elements also constitute a single $\mathscr{D}$-class of $P_{B}^{\bullet}$. 

It is immediate that the semigroup is $E^{\ast}$-unitary and combinatorial.

There are no atoms.
Let $xx^{-1}$ be any non-zero idempotent.
Then $x$ is a path from the vertex $v$, at level $n$, to the root $v_{0}$.
Let $e$ be an edge at level $n+1$ that ends in $v$.
from the definition of a Bratteli diagram, such an edge $e$ exists.
Then $xe(xe)^{-1} \leq xx^{-1}$.
\end{proof}

The following is immediate by the above.
We state it explicitly since it will be important later.

\begin{corollary} 
The natural partial order of a Bratteli inverse monoid is unambiguous.
\end{corollary}

\begin{remark} {\em It is possible to characterize Bratteli inverse monoids abstractly.
We do not do this here.}
\end{remark}

%%%%%%%%%%%%%%%%%%%%%%%%%%%%%%%%%%%%%%%%%%%%%%%%%%%%%%%%%%%%%%%%%%%%%%%%%%%%%%%%%%%%%%%%%%%%%%%%%%%%%%%
\subsection{Tight completions}

We now have two inverse monoids associated with a Bratteli diagram $B$:
the Bratteli inverse monoid $P_{B}^{\bullet}$ and the AF inverse monoid $\mathsf{I}(B)$.
Our goal is to explain how they are related.
We shall do this in the next section.
Here,  we describe the theory of essential completions of inverse semigroups.
This is described in \cite{LL1,LL2} and is based on ideas that generalize constructions to be found in Section~5 of \cite{Lenz} as well as in \cite{Law5}
and are related to those to be found in \cite{Exel1,Exel2}.
Our presentation here, though, is based on \cite{LL1,LL2} but we also take the opportunity to clarify some aspects of the theory developed there.

%%%%%%%%%%%%%%%%%%%%%%%%%%%%%%%%%%%%%%%%%%%%%%%%%%%%%%%%%%%%%%%%%%%%%%%%%%%%%%%%%%%%%%%%%%%%%%%%
Let $S$ be an inverse monoid with zero.
Given elements $a,b \in S$ such that $b \leq a$, 
we say that  {\em $b$ is essential in $a$} or that {\em $b$ is essentially contained in $a$}
if for each $0 \neq x \leq a$, the meet $x \wedge b \neq 0$.
The following motivates the definition.

\begin{lemma}\label{le:motivation} Let $S$ be a Boolean inverse monoid.
Let $a,b \in S$ such that $b \leq a$.
Then if $b$ is essentially contained in $a$ then $b = a$.
\end{lemma}
\begin{proof} We prove first that $b^{-1}b$ is essentially contained in $a^{-1}a$.
Let $0 \neq e \leq a^{-1}a$.
Then $ae \leq a$ and $ae \neq 0$.
Thus $ae \wedge b \neq 0$.
But since $ae,b \leq a$ they are compatible and so $\mathbf{d}(ae \wedge b) = \mathbf{d}(ae) \wedge \mathbf(d)(b)$.
It follows that $e \wedge b^{-1}b \neq 0$, as required.

Now suppose that $f$ is an idempotent essentially contained in $e$.
Then $e = e \vee f\overline{e}$ and $f \wedge e \overline{f} = 0$.
Thus $e\overline{f} = 0$ and so $e = f$.

Using this argument, we see that $b^{-1}b = a^{-1}a$ and so $b = a$, as claimed.
 \end{proof}

We now extend the definition from individual elements to finite subsets.
A finite subset $\{a_{1}, \ldots, a_{m} \} \subseteq a^{\downarrow}$ is said to be an {\em  (essential)  cover of $a$} 
if for each $0 \neq x \leq a$ we have that $x \wedge a_{i} \neq 0$ for some $i$.
We shall write $A \preceq a$ to mean $A$ is an (essential) cover of $a$.
Since the only covers to be considered in this paper are essential ones we shall simply say {\em cover} from now on.
The notions of an essential element and an essential subset are related.

\begin{lemma}\label{le:sputnik} \mbox{}
\begin{enumerate}

\item  Let $S$ be a distributive inverse semigroup.
Then $\{a_{1}, \ldots, a_{m}\} \preceq a$ if and only if $\bigvee_{i=1}^{m}a_{i} \preceq a$.

\item  Let $S$ be an inverse semigroup with zero.
Then $\{a_{1}, \ldots, a_{m} \} \preceq a$ if, and only if, $\{a_{1}, \ldots, a_{m}\}^{\downarrow} \preceq a^{\downarrow}$ in $\mathsf{D}(S)$.

\end{enumerate}
\end{lemma}
\begin{proof} (1) 
Let $0 \neq x \leq a$.
By part (3) of Lemma~\ref{lem:meets-joins},
we have that 
$$x \wedge \left( \bigvee_{i=1}^{m}a_{i} \right) =  \bigvee_{i=1}^{m} x \wedge a_{i}.$$
Suppose that $\{a_{1}, \ldots, a_{m}\} \preceq a$.
Then $x \wedge a_{i} \neq 0$ for some $i$.
It follows that $x \wedge \left( \bigvee_{i=1}^{m}a_{i} \right) \neq 0$.
Conversely, suppose that $x \wedge \left( \bigvee_{i=1}^{m}a_{i} \right) \neq 0$.
Then  $x \wedge a_{i} \neq 0$ for some $i$.

(2) Straightforward,
\end{proof}

The proofs of the following are all straightforward.

\begin{lemma}\label{le:essential_stuff} Let $S$ be an inverse semigroup with zero.
\begin{enumerate}

%1
\item The relation $\preceq$ is a partial order.

%2
\item $b \preceq a$ implies that $b^{-1} \preceq a^{-1}$.

%3
\item $b \preceq a$ and $d \preceq c$ implies that $bd \preceq ac$.

%4
\item $0 \preceq a$ implies that $a = 0$.

%5
\item If $b \preceq a$ then $\mathbf{d}(b) \preceq \mathbf{d}(a)$.

%6
\item Let $b \leq a$. Then $b \preceq a$ if and only if $\mathbf{d}(b) \preceq \mathbf{d}(a)$.

%7
\item Let $b, c \preceq a$. Then $b \wedge c \preceq a$.

%8
\item Suppose that $S$ is a $\wedge$-semigroup. If $a \preceq b$ and $c \preceq d$ then $a \wedge b \preceq b \wedge d$.

%9
\item Suppose that $S$ is a distributive inverse semigroup. If $a \preceq b$ and $c \preceq d$,
and in addition $a \sim c$ and $b \sim d$ then $a \vee b \preceq b \vee d$.

\end{enumerate}
\end{lemma}

If $A$ is a subset of $S$ then we define $\mathbf{d}(A) = \{ \mathbf{d}(a) \colon a \in A \}$.
We also have the following.

\begin{lemma}\label{le:coverage} Let $S$ be an inverse semigroup with zero.
\begin{enumerate}

\item $\{a\} \preceq a$.

\item $A \preceq a$ implies that $A^{-1} \preceq a^{-1}$.

\item $A \preceq a$ and $B \preceq b$ imply that $AB \preceq ab$.

\item If $X \preceq a$ and $X_{i} \preceq x_{i}$ for each $x_{i} \in X$ then $\bigcup_{i} X_{i} \preceq a$.
This is called {\em transitivity} of covers.

\item If $A \preceq a$ then $\mathbf{d}(A) \preceq \mathbf{d}(a)$.

\item Let $A \subseteq a^{\downarrow}$. Then $A \preceq a$ if and only if $\mathbf{d}(A) \preceq \mathbf{d}(a)$.

\item Let $A,B \preceq a$.
Define $A \wedge B = \{a' \wedge b' \colon a' \in A, b' \in B\}$.
Then $A \wedge B \preceq a$ and $A \wedge B = A \mathbf{d}(B) = B \mathbf{d}(A)$.

\item Let $S$ be an inverse $\wedge$-semigroup.
If $A \preceq a$ and $B \preceq b$ then $A \wedge B \preceq a \wedge b$.

%\item Let $A,B \preceq a$ and $B \preceq b$.
%Then $A \wedge B \preceq b$.

\end{enumerate}
\end{lemma}

The following was introduced in \cite{Lenz}.
Let $S$ be an arbitrary inverse semigroup with zero.
Let $a \in S$ and let $\{a_{1}, \ldots, a_{m}\}$ be a non-empty finite subset.
We write $a \rightarrow \{a_{1}, \ldots, a_{m}\}$ if for each $0 \neq x \leq a$ 
we have that $x^{\downarrow} \cap a_{i}^{\downarrow} \neq 0$ for some $i$.
In the case of $a \rightarrow \{b\}$, we simply write $a \rightarrow b$.
We write 
$\{a_{1}, \ldots, a_{m} \} \rightarrow \{b_{1}, \ldots, b_{n} \}$
iff $a_{i} \rightarrow  \{b_{1}, \ldots, b_{n} \}$ for $1 \leq i \leq m$.
We write 
$\{a_{1}, \ldots, a_{m} \} \leftrightarrow \{b_{1}, \ldots, b_{n} \}$
iff
$\{a_{1}, \ldots, a_{m} \} \rightarrow \{b_{1}, \ldots, b_{n} \}$
and
$\{b_{1}, \ldots, b_{n} \} \rightarrow \{a_{1}, \ldots, a_{m} \}$.
The proof of the following is immediate.

\begin{lemma} 
Let $\{a_{1}, \ldots, a_{m}\} \subseteq a^{\downarrow}$.
Then 
$\{a_{1}, \ldots, a_{m}\} \preceq a$ if, and only if, $\{a_{1}, \ldots, a_{m}\} \rightarrow a$.
\end{lemma}

%%%%%%%%%%%%%%%%%%%%%%%%%%%%%%%%%%%%%%%%%%%%%%%%%%%%%%%%%%%%%%%%%%%%%%%%%%%%%%%%%%%%%%
We now come to the key definition.
A homomorphism $\theta \colon S \rightarrow T$ from an inverse semigroup $S$ to a distributive inverse semigroup $T$
is said to be {\em tight}  if for each $a \in S$ and $A \preceq a$ we have that
$\theta (a) = \bigvee_{a_{i} \in A} \theta (a_{i})$.
Thus a tight homomorphism converts covers to joins.
A {\em tight completion} of $S$ is a distributive inverse semigroup $\mathsf{D}_{t}(S)$
together with a tight homomorphism $\tau \colon S \rightarrow \mathsf{D}_{t}(S)$
which is universal.
If such a completion exists then it is, of course, unique up to isomorphism.
We shall show that the essential completion exists.

\begin{remark} {\em Let $S$ be a distributive inverse semigroup.  
If $a = \bigvee_{i=1}^{m}a_{i}$ then $\{a_{1}, \ldots, a_{m}\}$ is a cover of $a$.
It follows that tight maps between distributive inverse semigroups 
preserve any finite joins that exist.
Thus they are morphisms of distributive inverse semigroups.}
\end{remark}

A morphism $\theta \colon S \rightarrow T$ between distributive inverse semigroups is said to be {\em essential}
if $x \preceq s$ implies that $\theta (x) = \theta (s)$.

\begin{lemma} 
A morphism $\theta \colon S \rightarrow T$ between distributive inverse semigroups is essential if and only if it is tight.
\end{lemma}
\begin{proof} Suppose that $\theta$ is essential.
Let $\{a_{1}, \ldots, a_{m} \} \subseteq a^{\downarrow}$ be a cover.
Then $b = \bigvee_{i=1}^{m} a_{i}$ is essential in $a$ by Lemma~\ref{le:sputnik}.
By assumption $\theta (a) = \theta (b)$.
We now use the fact that $\theta$ is a morphism and so $\theta (b) = \bigvee_{i=1}^{m} \theta (a_{i})$
which gives $\theta (a) = \bigvee_{i=1}^{m} \theta (a_{i})$, as required.
The proof of the converse is immediate. 
\end{proof}

It follows that in the case of distributive inverse monoids the word {\em tight} may be replaced by the word {\em essential}.

%%%%%%%%%%%%%%%%%%%%%%%%%%%%%%%%%%%%%%%%%%%%%%%%%%%%%%%%%%%%%%%%%%%%%%%%%%%%%%%%%%%%%%%%%
We begin by constructing the essential completion of a distributive inverse semigroup.
Let $S$ be a distributive inverse monoid.
Define the relation $\equiv$ on $S$ as follows
$$a \equiv b \Longleftrightarrow z \preceq a,b$$
for some $z \in S$.
We prove below that this is a congruence.
The $\equiv$-class containing $a$ is denoted by $[a]$.

\begin{lemma} Let $S$ be a distributive inverse semigroup.
The relation $\equiv$ is a $0$-restricted congruence on $S$.
If $S$ is $\wedge$-semigroup so is $S/\equiv$.
\end{lemma}
\begin{proof} The fact that the relation is an equivalence relation, a congruence and $0$-restricted
all follow from Lemma~\ref{le:essential_stuff}.

It remains to prove that if $S$ is an inverse $\wedge$-semigroup so is $S/ \equiv$.
Consider the elements $[a]$ and $[b]$.
By assumption, $a \wedge b$ exists and $[a \wedge b] \leq [a],[b]$.
Now let $[z] \leq [a],[b]$.
Then $z \equiv az^{-1}z \equiv bz^{-1}z$.
Let $u \preceq z, az^{-1}z$ and $v \preceq z,bz^{-1}z$.
Then $u \wedge v \preceq z,az^{-1}z,bz^{-1}z$
by Lemma~\ref{le:essential_stuff}.
It follows that $[z] = [az^{-1}z \wedge bz^{-1}z] \leq [a \wedge b]$, as required.
\end{proof}

\begin{proposition}\label{prop:one} Let $S$ be a distributive inverse semigroup.
Suppose that $e \preceq a$, where $e$ is an idempotent, implies that $a$ is an idempotent.
Then $\equiv$ is idempotent-pure  and $S/ \equiv$ is the essential completion of $S$.
\end{proposition}
\begin{proof} It is immediate from our assumption that $\equiv$ is idempotent-pure.
Put $U = S/ equiv$.
We show first that $U$ is a distributive inverse semigroup.
Suppose that $[a] \sim [b]$.
Then both $ab^{-1}$ and $a^{-1}b$ are idempotents because $\equiv$ is idempotent-pure and so $a \sim b$.
It follows that $a \vee b$ exists in $S$.
Since $a,b \leq a \vee b$ we have that $[a],[b] \leq [a \vee b]$.
Let $[a],[b] \leq [c]$.
Then $a \equiv ca^{-1}a$ and $b \equiv cb^{-1}b$.
By definition, we have that for some $x$ we have $x \preceq a$ and $x \preceq ca^{-1}a$.
Similarly, we have that for some $y$ we have that $y \preceq b$ and $y \preceq cb^{-1}b$.
It follows that $x \sim y$ since $x,y \leq a \vee b$.
Thus by Lemma~\ref{le:essential_stuff},
$x \vee y \preceq a \vee b$ and $x \vee y \leq c$.
Hence $[a \vee b] \leq [c]$, as required.
It is now straightforward to check that $U$ is a distributive inverse semigroup.

To prove that $\nu \colon S \rightarrow U$ is an essential map,
observe that $b \preceq a$ implies that $b \Leftrightarrow a$ and so $[b] = [a]$.

Let $\theta \colon S \rightarrow T$ be an essential homomorphism to a distributive inverse semigroup.
Denote the natural map from $S$ to  $S/ \equiv$ by $\nu$.
Define 
$$\bar{\theta} \colon S/\equiv  \longrightarrow T
\mbox{ by }
\bar{\theta}([a]) = \theta (a).$$
We show first that $\bar{\theta}$ is well-defined.
Suppose that $[a] = [b]$.
Then $a \equiv b$.
It follows that there is an element $x$ such that $x \preceq a$ and $x \preceq b$.
By assumption, $\theta (x) = \theta (a)$ and $\theta (x) = \theta (b)$.
It follows that $\theta (a) = \theta (b)$.
We need to prove that $\bar{\theta}$ preserves any binary joins that exist.
But this we have essentially done above.
By construction, we have that $\bar{\theta} \nu = \theta$
and it is immediate that $\bar{\theta}$ is unique with these properties.
\end{proof}

%%%%%%%%%%%%%%%%%%%%%%%%%%%%%%%%%%%%%%%%%%%%%%%%%%%%%%%%%%%%%%%%%%%%%%%%%%%%%%%%
We now describe how to construct the tight completion.

\begin{proposition}\label{prop:two} Let $S$ be an inverse semigroup.
Let $\theta \colon S \rightarrow T$ be a tight homomorphism to a distributive inverse semigroup.
Then the unique morphism $\theta^{\ast} \colon \mathsf{D}(S) \rightarrow T$ 
such that $\theta^{\ast}\iota = \theta$ is an essential morphism.
\end{proposition}
\begin{proof} It is enough to prove that $\theta^{\ast}$ is an essential map.
Let $A = \{a_{1}, \ldots, a_{m} \}^{\downarrow}$ and $B = \{b_{1}, \ldots, b_{n} \}^{\downarrow}$ be two elements of $\mathsf{D}(S)$
such that $B \preceq A$ in $\mathsf{D}(S)$.
We shall prove that $\theta^{\ast}(A) = \theta^{\ast}(B)$.
By definition $\theta^{\ast}(B) = \bigvee_{i=1}^{n} \theta (b_{i})$ and $\theta^{\ast}(A) = \bigvee_{j=1}^{m} \theta (a_{j})$.
Clearly $\theta^{\ast}(B) \leq \theta^{\ast}(A)$.
Thus by definition, we have that
$\theta^{\ast}(B) = \theta^{\ast}(A)(\bigvee_{j=1}^{n} \theta (\mathbf{d}(b_{j})))$.
We shall prove that $\theta (a_{i}) = \theta (a_{i})(\bigvee_{j=1}^{n} \theta (\mathbf{d}(b_{j})))$ from which the result follows
and to do that it is enough to prove that $a_{i} \rightarrow \{a_{i} \mathbf{d}(b_{1}), \ldots, a_{i} \mathbf{d}(b_{n})\}$.
Let $0 \neq z \leq a_{i}$.
Then $0 \neq z^{\downarrow} \leq A$.
It follows that there is a non-zero $C \in \mathsf{D}(S)$ such that $C \leq z^{\downarrow},B$.
We may therefore find $0 \neq c \in C$ such that $c \leq z, b_{j}$ for some $j$.
It follows that $0 \neq c \leq a_{i} \mathbf{d}(b_{j}),b_{j}$.
\end{proof}

%%%%%%%%%%%%%%%%%%%%%%%%%%%%%%%%%%%%%%%%%%%%%%%%%%%%%%%%%%%%%%%%%%%%%%%%%%%%%%%%%
Let $S$ be an arbitrary inverse semigroup with zero and let $a,b \in S$.
Define  $a \cong b$ if and only if $A \preceq a,b$ for some cover $A$.
We say that $a$ and $b$ {\em have a common cover}.

\begin{lemma} 
The relation $\cong$ is a $0$-restricted congruence.
If $a,b \leq c$ then $[a \wedge b] = [a] \wedge [b]$.
\end{lemma}
\begin{proof} 
The proof of the main claim follows from Lemma~\ref{le:coverage}.
We now prove the second claim.
Since $a,b \leq c$ we have that $[a],[b] \leq [c]$ and so $[a] \wedge [b]$ exists.
Clearly, $[a \wedge b] \leq [a] \wedge [b]$.
Let $[z] \leq [a],[b]$.
We have that $z \cong az^{-1}z \equiv bz^{-1}z$.
It follows that $z \cong (a \wedge b)z^{-1}z$, as required.
\end{proof}

We denote the $\cong$-class containing $a$ by $[a]$.
We say that an inverse semigroup is {\em separative} if $\cong$ is the equality relation.

\begin{lemma} Let $S$ be an inverse semigroup.
Then  $S/\cong$ is separative.
\end{lemma}
\begin{proof} Suppose that $[a] \cong [b]$.
Let $\{[x_{1}], \ldots, [x_{m}]\} \preceq [a], [b]$.
Put $e_{i} = \mathbf{d}(x_{i})$ for $1 \leq i \leq m$.
Then 
$\{[e_{1} ], \ldots, [e_{m} ]\} \preceq [\mathbf{d}(a)], [\mathbf{d}(b)]$.
Observe that $ae_{1}, \ldots, ae_{m} \leq a$ and that $ae_{i} \cong x_{i}$ for $1 \leq i \leq m$.
It is easy to check that 
$A = \{ae_{1}, \ldots, ae_{m} \} \preceq a$.
By a similar argument, 
$B = \{be_{1}, \ldots, be_{m}\} \preceq b$.
Now $ae_{i} \cong be_{i}$.
Let $C_{i} = \{c_{i1}, \ldots, c_{in}\}$ be a common cover of $ae_{i}$ and $be_{i}$.
We are interested in the set $\{c_{ij}\}$.
It is a subset of both $a^{\downarrow}$ and $b^{\downarrow}$.
Now $A$ is a cover of $a$
and for each element $ae_{i} \in A$ we have that $C_{i}$ is a cover of $ae_{i}$.
It follows by `transitivity of covers' that $C$ is a cover of $a$.
By symmetry, $C$ is a cover of $b$.
It follows that $a$ and $b$ have a common cover and so $a \cong b$.
\end{proof}

We denote by $\mu \colon S \rightarrow S/\cong$ the natural map.
The following is proved in \cite{LL2}.

\begin{proposition}\label{prop:three} Let $S$ be an inverse semigroup and let  $\theta \colon S \rightarrow T$ 
be a tight homomorphism to a distributive inverse semigroup $T$.
Then there is a unique tight homomorphism $\bar{\theta} \colon S/\cong \longrightarrow T$
such that $\bar{\theta} \mu = \theta$.
\end{proposition}

By the lemma above, we need only construct the tight completion in the separative case.
The lemma below will provide the connection with idempotent-purity.

\begin{lemma}\label{le:fred} Let $S$ be a separative inverse semigroup.
Let $A \preceq a$ where all elements of $A$ are idempotents.
Then $a$ is an idempotent.
\end{lemma} 
\begin{proof} Observe that $\mathbf{d}(A) = A$ and so $A \preceq \mathbf{d}(a)$.
It follows that $a \cong \mathbf{d}(a)$.
Thus $a = \mathbf{d}(a)$ by separativity.
\end{proof}

\begin{lemma} Let $S$ be an inverse semigroup and 
$A = \{a_{1}, \ldots, a_{m} \}^{\downarrow} \subseteq B = \{b_{1}, \ldots, b_{n} \}^{\downarrow}$
both be elements of $\mathsf{D}(S)$.
Then $A \preceq B$ if and only if $\{b_{1}, \ldots, b_{n} \} \rightarrow \{a_{1}, \ldots, a_{m}\}$.
\end{lemma}
\begin{proof} Suppose first that $A \preceq B$.
Let $0 \neq x \leq b_{i}$.
Then $0 \neq x^{\downarrow} \leq B$.
By assumption there exists $0 \neq C \leq x^{\downarrow}, A$.
Let $0 \neq c \in C$.
Then $c \leq x, a_{j}$ for some $j$.
To prove the converse, suppose that
 $\{b_{1}, \ldots, b_{n} \} \rightarrow \{a_{1}, \ldots, a_{m}\}$.
Let $0 \neq C \leq B$ where $C = \{c_{1}, \ldots, c_{p} \}^{\downarrow}$.
By assumption, for each $k$ we may find $x_{k}$ such that $0 \neq x_{k} \leq c_{k}, a_{i_{k}}$.
Put $X = \{x_{1}, \ldots, x_{p} \}^{\downarrow}$.
Then $X \neq 0$, $X \leq C$ and $X \leq A$.
\end{proof}         
 
We now have the following.

\begin{lemma}\label{le:daisy} Let $S$ be a separative inverse semigroup.
Then the congruence $\equiv$ is idempotent-pure on $\mathsf{D}(S)$.
\end{lemma}
\begin{proof} Let $A = \{a_{1}, \ldots, a_{m} \}^{\downarrow}$ and $B = \{b_{1}, \ldots b_{n}\}^{\downarrow}$.
Then $A \equiv B$ if and only if there exists $C \in \mathsf{D}(S)$ such that $C \preceq A$ and $C \preceq B$.
Suppose that $B$ is an idempotent.
Then $C$ is an idempotent.
But  $C \preceq A$.
Let $C = \{e_{1}, \ldots, e_{p}\}^{\downarrow}$ where the $e_{i}$ are idempotents.
We use the fact that $C$ is a compatible order ideal.
It follows that for each $a_{j}$ we have that $C \wedge a_{j}^{\downarrow}$ is defined and that
$C \wedge a_{j}^{\downarrow} \preceq a_{j}^{\downarrow}$.
But all the elements of $C \wedge a_{j}^{\downarrow}$ are idempotents and so by 
Lemma~\ref{le:fred}, we have that $a_{j}$ is an idempotent.
Thus $A$ is an idempotent, as required.
\end{proof}

We now have all the elements needed for our first main theorem.

\begin{theorem}[Tight completion] Let $S$ be an inverse semigroup with zero.
Let $\mathbf{S}$ be its separative image.
Then $\mathsf{D}_{t}(S) = \mathsf{D}(\mathbf{S})/\equiv$.
\end{theorem}
\begin{proof} Let $\theta \colon S \rightarrow T$ be a tight map to a distributive inverse semigroup.
Then by Proposition~\ref{prop:three}, we may replace $S$ by its separative image $\mathbf{S}$.
By Proposition~\ref{prop:two}, we may replace $\mathbf{S}$ by its Schein completion $\mathsf{D}(\mathbf{S})$.
But by Lemma~\ref{le:daisy}, the congruence $\equiv$ is idempotent-pure on $\mathsf{D}(\mathbf{S})$.
We may now apply Proposition~\ref{prop:one} to get the result.
\end{proof}

%%%%%%%%%%%%%%%%%%%%%%%%%%%%%%%%%%%%%%%%%%%%%%%%%%%%%%%%%%%%%%%%%%%%%%%%%%%%%%%%%%%%%
The following will apply to the application of tight completions that will interest us.

\begin{lemma}\label{le:binky} Let $S$ be an inverse monoid with zero.
\begin{enumerate}

\item If $S$ is a $\wedge$-monoid then $\{a_{1}, \ldots, a_{m}\}^{\downarrow} \equiv \{b_{1}, \ldots, b_{n}\}^{\downarrow}$ 
if, and only if, $\{a_{1}, \ldots, a_{m}\} \leftrightarrow \{b_{1}, \ldots, b_{n} \}$.

\item If $S$ is $E^{\ast}$-unitary, then $\equiv$ is idempotent-pure on $\mathsf{D}(S)$.

\end{enumerate}
\end{lemma}
\begin{proof} (1) Suppose that
$A = \{a_{1}, \ldots, a_{m}\}^{\downarrow} \equiv \{b_{1}, \ldots, b_{n}\}^{\downarrow} = B$.
Then there exists 
$C = \{c_{1}, \ldots, c_{p}\}^{\downarrow}$ such that
$\{c_{1}, \ldots, c_{p}\}^{\downarrow} \preceq \{a_{1}, \ldots, a_{m}\}^{\downarrow}$
and 
 $\{c_{1}, \ldots, c_{p}\}^{\downarrow} \preceq \{b_{1}, \ldots, b_{n}\}^{\downarrow}$.
Let $0 \neq x \leq a_{i}$.
Then $0 \neq x^{\downarrow} \leq A$.
Thus $x^{\downarrow} \wedge C \neq 0$.
It follows that $x \wedge c_{k} \neq 0$ for some $k$.
But $c_{k} \leq b_{j}$ for some $j$.
Thus $x \wedge b_{j} \neq 0$ for some $j$.
By symmetry, it follows that  $\{a_{1}, \ldots, a_{m}\} \leftrightarrow \{b_{1}, \ldots, b_{n} \}$.

Conversely, suppose that $\{a_{1}, \ldots, a_{m}\} \leftrightarrow \{b_{1}, \ldots, b_{n} \}$.
Put $C= \{a_{i} \wedge b_{j} \colon 1 \leq i \leq m, 1 \leq j \leq n \}^{\downarrow}$.
It is now easy to check that $C \preceq A,B$.
It follows that $A \equiv B$, as required.

(2) Let $A,B \in \mathsf{D}(S)$ be non-zero elements such that $A$ is an idempotent and $A \equiv B$.
We prove that $B$ is an idempotent.
By definition, there is $C$ such that $C \preceq A,B$.
Since $C \subseteq A$, all the elements of $C$ are idempotents.
Let $b$ be one of the generators of $B$.
Then $b^{\downarrow} \leq B$.
Thus $b^{\downarrow}$ has a non-zero intersection with some element $e^{\downarrow}$ where $e \in C$.
It follows that $b$ is above a non-zero idempotent and so itself an idempotent, as required.
\end{proof}

%%%%%%%%%%%%%%%%%%%%%%%%%%%%%%%%%%%%%%%%%%%%%%%%%%%%%%%%%%%%%%%%%%%%%%%%%%%%%%%%%%%%%%%%%%%%%%%%%%%%%%%%%%%%%%%%%
\subsection{Tight completions of Bratteli inverse monoids are AF}

The goal of this section is to prove the following.

\begin{theorem} Let $B$ be a Bratteli diagram.
Then  $\mathsf{I}(B)$ is the tight completion of  $P_{B}^{\bullet}$.
\end{theorem}

Since $P_{B}^{\bullet}$ is $E^{\ast}$-unitary, 
it follows by Lemma~\ref{le:binky}, that we need to show $\mathsf{I}(B)$ is isomorphic to $\mathsf{D}(P_{B}^{\bullet})/\equiv$.
By Lemma~\ref{le:unambig},
Proposition~\ref{prop:boozy}
and
Lemma~\ref{le:plod},
the inverse monoid $P_{B}^{\bullet}$
has an unambiguous natural partial order and so every finitely generated
compatible order ideal is generated by pairwise orthogonal elements.

We now describe, informally, why $\mathsf{D}(P_{B}^{\bullet})/\equiv$ is isomorphic to $\mathsf{I}(B)$.
Let $v$ be a vertex of the Bratteli diagram $B$.
Denote by $P_{v}$ all paths from $v$ to the root.
Denote by $G_{v}$ the groupoid of all elements $xy^{-1}$ where $x,y \in P_{v}$.
These are just the non-zero $\mathscr{D}$-classes and so $P_{B}^{\bullet}$ is a disjoint union of them.
Think of these groupoids as attached to the appropriate vertices of the Bratteli diagram.
Elements at one level can only be above, in the natural partial order, elements at lower levels,
this order being mediated by edges since $xx^{-1} \leq yy^{-1}$ if, and only if, $x = yp$ for some path $p$.
The key point is that if we adjoin a zero to $G_{v}$ and take its distributive completion, we get a symmetric inverse monoid with letters the set $P_{v}$.
In particular, a symmetric inverse monoid of the correct size associated with the vertex $v$ of the Bratteli diagram $B$.
The disjoint union of groupoids at one level is therefore inflated to a direct product of symmetric inverse monoids of the correct sizes
for that level of the Bratteli diagram.
The standard maps between adjacent levels are induced by the natural partial order linking elements at one level to the elements immediately below it.
The elements lower down are essential in the elements immediately above and so are identified in the tight completion.
It remains to convert this informal description into a detailed proof.

An element of $\mathsf{D}(P_{B}^{\bullet})$ is said to be {\em ($p$-)homogeneous} if it is generated by pairwise orthogonal elements that all have the same weight $p$.
Define $\mathsf{D}^{h}(P_{B}^{\bullet})$ to be the subset of $D$ consisting of all homogeneous elements of $D$ and zero.
If two elements of weight $p$ are multiplied together the result is either zero or an element of weight $p$.
It follows that $\mathsf{D}^{h}(P_{B}^{\cdot})$ forms an inverse submonoid of $\mathsf{D}(P_{B}^{\bullet})$ that we call the {\em homogeneous orthogonal completion}.

We now show how (tight) covers of idempotents in $P_{B}^{\bullet}$ may arise.

\begin{lemma}[Lengthening]\label{le:lengthening} Let $x$ be a path from the vertex $v$, at level $n$, to the root $v_{0}$.
Let $e_{1}, \ldots, e_{q}$ be all the edges at level $n+1$ that end in $v$.
Then the set $A = \{ xe_{1}(xe_{1})^{-1}, \ldots, xe_{q}(xe_{q})^{-1}\}$ is a cover of $xx^{-1}$,
and all the elements of $A$ have length $n+1$.
\end{lemma}
\begin{proof} By the definition of a Bratteli diagram, there is at least one such edge.
Any non-zero idempotent $yy^{-1} \leq xx^{-1}$ must be such that $y = xe_{i}p$ for some edge $e_{i}$ and path $p$.
Thus $yy^{-1} \leq xe_{i}(xe_{i})^{-1}$.
\end{proof}

The above process may be iterated.
Observe that, from the definition of a Bratteli diagram, a path of length $p$ can always be lengthened
to a path of length $p+1$.
The proof of the following is now almost immediate.

\begin{lemma}[Homogenizing]\label{le:homogenizing} Every inhomogeneous element of  $\mathsf{D}(P_{B}^{\bullet})$ is $\equiv$-related to a homogeneous one.
\end{lemma}

\begin{lemma}\label{le:distinctions} Let $A$ and $B$ be two $p$-homogeneous elements of $\mathsf{D}(P_{B}^{\bullet})$.
Then $A \equiv B$ implies that $A = B$.
\end{lemma} 
\begin{proof} The result follows from the following observations.
We are working in an inverse monoid with an unambiguous natural partial order.
Two elements which have a non-zero lower bound must be comparable.
But comparable elements of the same weight must be equal.
\end{proof}

Thus the structure of $\mathsf{D}(P_{B}^{\bullet})/\equiv$ will be strongly influenced by the structure of  $\mathsf{D}^{h}(P_{B}^{\bullet})$. 
In addition, the congruence $\equiv$ cannot identity two $p$-homogeneous elements.

Consider now level $p$ of the Bratteli diagram $B$.
Let the vertices be $v_{1}, \ldots, v_{k}$.
Let the sizes of these vertices be $m_{1}, \ldots, m_{k}$, respectively.
Then the $p$-homogeneous elements are in bijective correspondence with
the non-zero  elements of $I_{m_{1}} \times \ldots \times I_{m_{k}}$.
This is because the distributive completion of the groupoid $G_{v_{j}}$ with a zero adjoined is just $I_{m_{j}}$.
Thus the $p$-homogeneous elements of  $\mathsf{D}^{h}(P_{B}^{\bullet})$ with an adjoined zero form the {\em correct} semisimple inverse monoid.
We now describe how successive levels are related.

For convenience, we put $S = P_{B}^{\bullet}$.
Let $s \in S$ be any non-zero element.
Define
$$\varepsilon (s) = \{t \in S \colon t \leq s \mbox{ and } \beta (t) = \beta (s) + 1\}.$$
This set consists of all elements {\em immediately below} $s$.
This is a finite set since $\mu^{-1}(p)$ is finite for any $p$.
This is a compatible subset because all elements are bounded above by $s$.
This is an orthogonal set  since two elements which are of the same weight and compatible are either equal or orthogonal.
Define $\varepsilon (0) = \{0\}$.

\begin{lemma}\label{le:essential} \mbox{}
\begin{enumerate}

\item $\varepsilon (st)^{\downarrow} = \varepsilon (s)^{\downarrow} \varepsilon (t)^{\downarrow}$.

\item $\varepsilon (s \wedge t)^{\downarrow} = \varepsilon (s)^{\downarrow} \wedge \varepsilon (t)^{\downarrow}$.

\item $\varepsilon (s)^{\downarrow} \preceq s^{\downarrow}$.

\end{enumerate}
\end{lemma}
\begin{proof} (1) We suppose that $st \neq 0$ and that $\beta (s) \geq \beta (t)$, without loss of generality.
Clearly, 
$\varepsilon (s)^{\downarrow} \varepsilon (t)^{\downarrow} \subseteq \varepsilon (st)^{\downarrow}$
since multiplying two elements together is never weight-decreasing.
Let $a \leq st$ of weight at least $\beta (s) + 1$.
Then $0 \neq at^{-1} \leq stt^{-1} = s$.
Thus $at^{-1} \in \varepsilon (s)^{\downarrow}$.
Hence $a \in \varepsilon (s)^{\downarrow} (s^{-1}st)$ and $s^{-1}st \in \varepsilon (t)^{\downarrow}$.

(2) This is immediate.

(3) immediate.

\end{proof}

Define $\eta (s) = \varepsilon (s)^{\downarrow}$.
Put $\mathbf{e}_{1} = \varepsilon (1)^{\downarrow}$.
Then we have defined a morphism $\eta \colon S \rightarrow \mathbf{e}_{1}\mathsf{D}(S)\mathbf{e}_{1}$.
Clearly, $e\mathsf{D}(S)e$ is a distributive inverse monoid and so
the map extends uniquely to a map $\eta \colon \mathsf{D}(S) \rightarrow \mathbf{e}_{1}\mathsf{D}(S)\mathbf{e}_{1}$.
This map is injective essentially by unambiguity.
We therefore obtain a strictly decreasing sequence of idempotents
$\mathbf{e}_{1}  >  \mathbf{e}_{2}  >   \mathbf{e}_{3}  >  \ldots$.
It follows that, by restriction, we get injective morphisms
$\eta_{p} \colon  \mathbf{e}_{p}\mathsf{D}(S)\mathbf{e}_{p} \rightarrow  \mathbf{e}_{p+1}\mathsf{D}(S)\mathbf{e}_{p+1}$
between local submonoids.
The elements of maximum weight in $\mathbf{e}_{p}\mathsf{D}(S)\mathbf{e}_{p}$ consist of the $p$-homogenous elements of $\mathsf{D}(S)$.
Denote these by $D_{p}$.
It follows that we get a map $\varepsilon_{p} \colon D_{p}  \rightarrow D_{p+1}$ by restriction.
We now calculate what this map does.
Let the vertices at level $p$ be $v_{1}^{p}, \ldots, v_{k}^{p}$.
Let $(\mathbf{s}_{1}, \ldots, \mathbf{s}_{k})$ be a $k$-tuple of $p$-homogeneous elements.
Where $\mathbf{s}_{j} \in G_{v_{j}^{p}}$.  
We calculate the $i$th term of the image of this element under $\varepsilon_{p}$.
There are $s_{ij}$ edges joining vertex $v_{j}^{p}$ and vertex $v_{i}^{p+1}$.
We denote these edges by $e_{ij}^{\nu}$.

We now make a simple observation.
If $xy^{-1} \in G_{v}$ and $e$ is an edge that connects $v$ and $v'$, where $v'$ is a vertex at level $p+1$.
Then $(xe)(ye)^{-1} \leq xy^{-1}$ and belongs to $G_{v'}$.
We say that $(xe)(ye)^{-1}$ is obtained from $xy^{-1}$ by an {\em edge-adjunction using the edge $e$}.

It follows that the effect of $\varepsilon_{p}$ is to carry out all possible edge-adjunctions
on the given $k$-tuple of elements of weight $p$ using all the edges joining level $p$ to level $p+1$. 
It follows that on non-zero elements the effect of $\varepsilon_{p}$ is that of the corresponding standard map.
The proof of the following is now almost immediate.

\begin{proposition} Let $B$ be a Bratteli diagram and let $S_{0} \rightarrow S_{1} \rightarrow S_{2} \rightarrow \ldots$ be
the associated sequence of semisimple monoids and injective morphisms.
Construct from this the $\omega$-chain of inverse monoids $S$ and factor out by the ideal $\mathscr{Z}$.
Then the homogeneous orthogonal completion of $P_{B}^{\bullet}$ is isomorphic to $S/\mathscr{Z}$.
\end{proposition}

We now describe the effect of glueing elements together under $\equiv$.
By Lemma~\ref{le:motivation}, this relation is trivial on Boolean inverse monoids.
It follows that it can only be non-trivial between the layers of $\mathsf{D}^{h}(P_{B}^{\bullet})$. 
By Lemma~\ref{le:essential},  if one element is mapped to another by $\varepsilon_{p}$ then they will be identified by $\equiv$
and the congruence $\equiv$ is generated by these identifications.

%%%%%%%%%%%%%%%%%%%%%%%%%%%%%%%%%%%%%%%%%%%%%%%%%%%%%%%%%%%%%%%%%%%%%%%%%%%%%%%%%%%%%%%%%%%%%%%%%%%%
\subsection{The associated groupoid}

From the theory described in \cite{Law5}, with each Boolean inverse $\wedge$-monoid we may associate a Hausdorff \'etale topological groupoid
under non-commutative Stone duality.
We shall describe the groupoid associated with $\mathsf{I}(B)$.
This looks like a difficult question.
However, we have proved that this inverse monoid is constructed from $P_{B}^{\bullet}$.
Again, by the theory described in \cite{Law5}, we will get the same groupoid if we start from the much simpler monoid $P_{B}^{\bullet}$.
The elements of the groupoid are the ultrafilters in  $P_{B}^{\bullet}$.
Since this inverse monoid is $E^{\ast}$-unitary,
the identities of the groupoid correspond to the idempotent ultrafilters and these are in bijective correspondence with the infinite paths in the
Bratteli diagram $B$ that end at the root.
The non-identity elements of the groupoid are essentially cosets,
and may be identified with triples $(yw,yx^{-1},xw)$ where $w$ is an infinite path in $B$ to the root and $x$ and $y$ are two paths that start from the same vertex and end at the root.
It follows that the groupoid is just {\em tail-equivalence} \cite{ER}.

%%%%%%%%%%%%%%%%%%%%%%%%%%%%%%%%%%%%%%%%%%%%%%%%%%%%%%%%%%%%%%%%%%%%%%%%%%%%%%%%%%%%%%%%%%

\end{document}